\documentclass[12pt]{amsart}

\usepackage{enumerate}
\usepackage{mathrsfs}
\usepackage{amsmath,amssymb,amsfonts}
\usepackage{newcent} 
\usepackage{graphicx}
\usepackage[normalem]{ulem}
\usepackage{fullpage}
\usepackage{esvect}
\usepackage{color}
\usepackage{tikz}
\usetikzlibrary{calc}
\usetikzlibrary{arrows.meta}
\usetikzlibrary{shapes.geometric}

\newtheorem{thm}{Theorem}[section]
\newtheorem{theorem}[thm]{Theorem}
\newtheorem{lemma}[thm]{Lemma}
\newtheorem{cor}[thm]{Corollary}
\newtheorem{corollary}[thm]{Corollary}

\newtheorem{proposition}[thm]{Proposition}

\newtheorem{question}[thm]{Question}

\theoremstyle{definition}
\newtheorem{defn}[thm]{Definition}

\theoremstyle{remark}

\makeatletter
\let\c@equation\c@thm
\makeatother
\numberwithin{equation}{section}

\newcommand{\size}[1]{\left \vert #1 \right \vert}
\newcommand{\cart}{\, \Box \,}

\newcommand{\floor}[1]{\left \lfloor #1 \right \rfloor}

\DeclareMathOperator{\dist}{\textrm{dist}}

\DeclareMathOperator{\tw}{tw}

\DeclareMathOperator{\dirloc}{\zeta_d}
\DeclareMathOperator{\dirlocinc}{\dirloc}
\DeclareMathOperator{\dirloccomp}{\zeta_d^*}

\begin{document}

\bibliographystyle{plain}

\title{The Directional Localization Game on Graphs}

\author{John Jones}
\address{Department of Mathematics and Applied Mathematical Sciences, University of Rhode Island, University of Rhode Island, Kingston, RI, USA, 02881}
\email{\tt jj\_jones@uri.edu}

\author{William B. Kinnersley}
\address{Department of Mathematics and Applied Mathematical Sciences, University of Rhode Island, University of Rhode Island, Kingston, RI, USA, 02881}
\email{\tt billk@uri.edu}


\begin{abstract}
In the localization game on a graph $G$, a team of cops searches for an invisible, mobile robber on $G$ by ``probing'' vertices; each probe tells the cops the distance from the probed vertex to the robber.  The cops win if they can uniquely determine the robber's location.  In this paper, we introduce a related game: the \textit{directional localization game}.  In this game, instead of probes returning distances, they return directions: when the cops probe a vertex $v$, the robber must respond with one or more neighbors of $v$ that lie on a shortest path from $v$ to the robber's location.  The minimum number of cops needed to win this game on $G$ is the \textit{directional localization number} of $G$.  We study the directional localization game on several classes of graphs, including chordal graphs, Cartesian products, and incidence graphs of projective planes.  We also bound the directional localization number of a graph $G$ in terms of the degeneracy and the treewidth of $G$. 
\end{abstract}

\maketitle

\section{Introduction}

\textit{Pursuit-evasion games} model scenarios in which a mobile \textit{evader} seeks to avoid one or more \textit{pursuers}.  In the discrete setting, one typically uses a graph as the ``game board'', with the evader moving between vertices of the graph.  Perhaps the best-known example of a pursuit-evasion game is \textit{Cops and Robbers}, introduced independently by Nowakowski and Winkler \cite{NW83} and by Quillot \cite{Qui78}.  Cops and Robbers is a two-player game that pits a team of one or more cops against a single robber.  All of the cops and the robber occupy vertices of some graph $G$, and the players take turns moving from vertex to adjacent vertex.  If any cop ever occupies the same vertex as the robber, then the cops win; conversely, the robber wins if he can evade capture indefinitely.  The key parameter of interest is the \textit{cop number} of $G$, which represents the minimum number of cops needed to win when the game is played on $G$.  For more background on Cops and Robbers, and for a sampling of the many variants of the game that have been studied, see the recent survey \cite{BN11}.

In the standard model of Cops and Robbers, both players have perfect information: the cops know the robber's location and vice-versa.  Recently, however, there has been significant interest in pursuit-evasion games in which the pursuers have only partial information about the evader's location.  For example, To\v{s}i\`{c} \cite{Tos85} introduced \textit{zero-visibility Cops and Robbers}, wherein the robber is ``invisible'' and cannot be detected by the cops until such time as they capture him; Clarke et al. \cite{CCDDFM20} later introduced the more general \textit{$\ell$-visibility Cops and Robbers}, wherein cops can only ``see'' the robber if he occupies a vertex within distance $\ell$ of some cop.  

Another such model of pursuit and evasion is the \textit{localization game}.  The localization game can be viewed as a variant of Cops and Robbers in which the cops do not know the robber's location, akin to zero-visibility Cops and Robbers.  However, in the localization game on a graph $G$, the cops do not actually traverse $G$.  Instead, in each round of the game, each cop chooses a vertex of $G$ to \textit{probe}, after which the cops learn the distance from the robber's location to each probed vertex.  If the cops have obtained enough information to uniquely determine the robber's location, then the cops win.  Otherwise, the game continues: the robber may move to a neighboring vertex, then the cops again choose vertices to probe, and so on.  If the cops have a strategy to guarantee that they locate the robber within a bounded number of rounds, then the cops win the game; otherwise, the robber wins.  The minimum number of cops needed to win the game on a graph $G$ is the \textit{localization number} of $G$, denoted $\zeta(G)$.  The localization game was first introduced by Seager \cite{Sea12}, who considered the game for a single cop; the game was later generalized to multiple cops by Haslegrave, Johnson, and Koch \cite{HJK18} and has been further studied since then (see e.g. \cite{BMMNP18, BGGNSS18, DFP19, KB20}).

In this paper, we introduce a new variant of the localization game, which we call the \textit{directional localization game}.  The directional localization game is played similarly to the localization game; the chief difference between the games is the information obtained by the cops' probes.  In the directional localization game, instead of returning the distance to the robber, each probe returns the direction: the robber responds to each probe with a set of one or more vertices that are adjacent to the probe and lie on some shortest path between the probe and the robber.  In essence, probes in the directional localization game serve as passive sensors that can determine the direction from which a signal originated, but not the range.  As with the localization game, in the directional localization game the cops win if and only if they have a strategy to determine the robber's location within a bounded number of rounds.  We define the minimum number of cops needed to win the game on a graph $G$ to be the \textit{directional localization number} of $G$.  For a more precise explanation of the game, refer to Section \ref{sec:defns}.


Our focus is on determining and bounding the directional localization numbers on several natural classes of graphs. The paper is laid out as follows. In Section \ref{sec:dirloc_prelim}, we provide a formal definition of the game and determine the directional localization numbers of several simple graph classes.  In Section \ref{sec:dirloc_genbounds}, we establish general bounds on the directional localization number in terms of other graph parameters; most notably, Theorem \ref{thm:degen_upper} states that a $k$-degenerate has directional localization number at most $k+1$, while Theorems \ref{thm:cartesian_complete} and \ref{thm:cartesian_lower} give bounds on the directional localization numbers of Cartesian products in terms of the directional localization numbers of their factors.  In Section \ref{sec:dirloc_treewidth}, we explore the connection between the directional localization number of a graph and its treewidth.  In Theorem \ref{thm:treewidth_restricted} we show that if a graph has a special kind of tree decomposition of width $k$, then its directional localization number is at most $k$.  As an application of this result, in Theorem \ref{thm:small_treewidth} we show that if a graph has treewidth at most 2, then its directional localization number is bounded above by its treewidth. In Section \ref{sec:dirloc_proj_planes}, we consider incidence graphs of projective planes, showing that such graphs always have directional localization number 2 (Theorem \ref{thm:proj_planes}).  Finally, in Section \ref{sec:dirloc_future_work}, we conclude the paper with some open questions and potential directions for future research.


\section{Preliminaries}\label{sec:dirloc_prelim}

\subsection{Definitions and Notation}\label{sec:defns}

The directional localization game is played on a connected graph $G$ between a team of \textit{cops} and a single \textit{robber}.  (Throughout this paper, unless otherwise specified, all graphs are assumed to be undirected and connected.)  For convenience, we will refer to a single cop using the pronouns she/her and to the robber using the pronouns he/him.  Throughout the game, the robber occupies a vertex of $G$.  At the beginning of the game, the robber chooses his initial location, and in the course of the game he will have the opportunity to move from vertex to vertex.  

Unlike many pursuit-evasion games, in the directional localization game, the cops do not know the robber's location.  Indeed, their goal is to uniquely determine the robber's location.  To accomplish this, the cops make a series of vertex \textit{probes} that yield information about the robber's location.  When a cop probes a vertex, the robber must issue a \textit{response} to that probe.  Intuitively, the robber's response indicates the direction from the probed vertex to the robber's location.  Formally, the robber's response to a cop probe at vertex $v$ is a nonempty subset $S$ of $N[v]$, where:
\begin{itemize}
\item $v \in S$ if and only if the robber is located at $v$, and 
\item for $w \not = v$, if $w \in S$, then some shortest path from $v$ to the robber's location passes through $w$.
\end{itemize}
If the cops probe $v$ and the robber's response includes some vertex $w$, we sometimes say that the probe at $v$ \textit{returns} $w$ or \textit{points at} $w$.  The cops share information; that is, whenever a cop probes a vertex, all cops know which vertex was probed and what the robber's response was.  After the cops probe a vertex $v$, we say that a vertex $w$ is \textit{consistent with the probe at $v$} if the robber's response to the probe does not preclude the possibility that the robber occupies $w$.  (Note that consistency with a probe does not take into account information obtained from other probes, previous rounds, etc.; it depends only on the robber's response to that single probe.)

The game's play is divided into a sequence of discrete \textit{rounds}.  Each round consists of two \textit{phases}:
\begin{itemize}
\item a \textit{probing phase} in which each cop selects a vertex of $G$ to probe and the robber simultaneously responds to all of the probes; and
\item a \textit{recontamination phase} in which the robber either remains in place or moves to a vertex adjacent to his current location.
\end{itemize}

If at any point during the game the cops have uniquely determined the robber's location, then the cops win the game.  Conversely, if the robber has a strategy to prevent this indefinitely, then the robber wins.  We require that the cops be able to guarantee determination of the robber's location in a bounded number of rounds; this prevents the cops from simply probing vertices at random and hoping to probe the robber's location through random chance.  Equivalently, we view the robber as omniscient, and we assume that he makes perfect moves armed with complete foreknowledge of the cops' actions. 

In this paper, we actually consider two different models of the directional localization game.  These models differ in the requirements placed on the robber's responses to the cops' probes.  In the \textit{partial-feedback} model, in response to a probe at $v$, the robber must provide \textbf{any} single vertex $w$ such that $w$ lies on some shortest path from $v$ to the robber's location.  In the \textit{full-feedback} model, we require more from the robber: his response to a probe at $v$ must contain \textbf{all} vertices in $N(v)$ that lie on a shortest path from $v$ to the robber's location.

When studying the game, we typically aim to determine, for a given graph $G$, the minimum number of cops needed to win the game on $G$.  Thus motivated, we define the \emph{partial-feedback directional localization number} of $G$, denoted $\dirloc(G)$, to be the minimum number of cops needed to win the game on $G$ in the partial-feedback model.  Similarly, the \emph{full-feedback directional localization number of $G$}, denoted $\dirloccomp(G)$, is the minimum number of cops needed to win in the full-feedback model.

As with many pursuit-evasion games with incomplete information, the directional localization game shares many similarities with classical graph searching games.  As such, we borrow some terminology from the realm of graph searching.  We say that a vertex is \emph{contaminated} at some point in time if that vertex could potentially contain the robber, taking into consideration all information the cops have received thus far.  A vertex that isn't contaminated is said to be \emph{clear}.  During the probing phase of a round, the cops may receive information that changes the state of a vertex from contaminated to clear; we refer to this as \emph{clearing} a vertex.  During the recontamination phase -- when the robber gets the opportunity to move to a new vertex -- vertices may change from clear to contaminated; we refer to this change as the \emph{recontamination} of those vertices.

The cops win the game when, and only when, the graph has exactly one contaminated vertex.  Thus, one could say that the cops' goal is to clear all but one vertex of the graph.  Viewing the game from this perspective, instead of playing against a single robber who moves from vertex to vertex, the cops can fight against a sea of contamination that spreads in all possible directions during each recontamination phase. This contamination simultaneously represents all possible movement strategies the the robber might employ. Thus, the robber's agency is limited to how he responds to each probe, and we may disregard the issue of how he moves throughout the graph; this often makes the game easier to analyze.

We sometimes refer to the set of all contaminated vertices at a given point in time as the \textit{robber territory}.  Given a graph $G$ and positive integer $k$, we denote by $R_{G,k}$ the robber territory at the start of round $k$ -- that is, prior to the probing phase -- and we denote by $R'_{G,k}$ the robber territory immediately after the probing phase of round $k$.  When the graph $G$ is clear from context, we sometimes write just $R_k$ and $R'_k$.\\

\subsection{Elementary Bounds}

In this subsection, we determine the directional localization numbers of several elementary classes of graphs.  Before this, however, we make an observation regarding the relationship between the partial-feedback and full-feedback models of the game.

\begin{proposition}\label{prop:complete_easier}
For every graph $G$, we have $\dirloccomp(G) \le \dirloc(G)$.
\end{proposition}
\begin{proof}
If $k = \dirloc(G)$, then $k$ cops can win the full-feedback game on $G$ by ignoring all but one (arbitrarily-chosen) vertex in the robber's response to each probe, and following a winning strategy for the partial-feedback game on $G$ based on these responses.
\end{proof}

We now determine the values of $\dirloc$ and $\dirloccomp$ on several elementary classes of graphs.  Before launching into the proof, we remind the reader that when a cop probes some vertex $v$, the probe returns $v$ itself if, and only if, the robber occupies $v$.  In other words, if the cops ever probe the robber's vertex, then the cops win immediately.

\begin{proposition}\label{prop:simple}\,
\begin{enumerate}
\item [(i)] $\dirloccomp(K_n) = \dirloc(K_n) = 1$ for all $n$.
\item [(ii)] $\dirloccomp(P_n) = \dirloc(P_n) = 1$ for all $n$.
\item [(iii)] $\dirloccomp(C_4) = 1$; $\dirloc(C_4) = 2$; and $\dirloccomp(C_n) = \dirloc(C_n) = 2$ for $n \ge 5$.
\item [(iv)] If $m \ge n$, then $\dirloccomp(K_{m,n}) = 1$ and $\dirloc(K_{m,n}) = n$.
\end{enumerate}
\end{proposition}
\begin{proof}\,
\begin{enumerate}
\item [\textbf{(i)}] In both models of the game, probing any single vertex uniquely determines the location of the robber.\\

\item [\textbf{(ii)}] By Proposition \ref{prop:complete_easier}, it suffices to show that $\dirloc(P_n) \le 1$.  To do this, we explain how a single cop can locate a robber on $P_n$ in the partial-feedback game.  The claim follows from part (i) if $n \le 2$, so suppose $n \ge 3$.  Let $v_1, v_2, \dots, v_n$ denote the vertices of $P_n$, in order.  The cop first probes $v_2$.  If the robber occupies $v_2$, then the cop clearly wins.  If instead the probe points at $v_1$, then the robber can only occupy $v_1$, so again the cop wins.  Finally, suppose that the probe points at $v_3$.  The robber must occupy a vertex in $\{v_3, \dots, v_n\}$.  After the robber's ensuing move, he must occupy a vertex in $\{v_2, \dots, v_n\}$.  The cop next probes $v_3$; similarly to before, if the robber occupies $v_2$ or $v_3$, then the cop wins.  Otherwise, after the robber's ensuing move, he must occupy a vertex in $\{v_3, \dots, v_n\}$.  The cop now probes $v_4$, and so on; by proceeding in this matter, the cop must eventually locate the robber and thereby win the game.\\

\item [\textbf{(iii)}] It is clear by inspection that $\dirloccomp(C_4) = 1$ and $\dirloc(C_4) = 2$.  For larger cycles, by Proposition \ref{prop:complete_easier} it suffices to argue that for $n \ge 5$, we have $\dirloccomp(C_n) > 1$ and $\dirloc(C_n) \le 2$.

We begin by arguing that $\dirloccomp(C_n) > 1$.  Consider the full-feedback game on $C_n$; we give a strategy whereby the robber can ensure that for all $k$, the set $R_k$ -- that is, the robber territory prior to the recontamination phase of round $k$ -- contains some four consecutive vertices along the cycle.  This guarantees that the robber can never be located: if the cop ever narrowed the robber's location down to a single vertex, then after the ensuing robber turn, the robber territory would only contain three vertices, not four.  

To prove this claim, we use induction on $k$.  When $k=1$, the claim is clearly true.  Fix some value of $k$, and suppose that $R_k$ contains four consecutive vertices along the cycle; we will show how the robber can ensure that the same is true of $R_{k+1}$.  Let the vertices of $C_n$ be labeled $v_0, v_1, \dots, v_{n-1}$ in order, and suppose by symmetry that $\{v_{n-2}, v_{n-1}, v_0, v_1\} \subseteq R_k$.  Suppose the cop probes some vertex $v_i$.  If $2 \le i < n/2$, then the robber responds with $v_{i-1}$; this response is consistent with both $v_0$ and $v_1$, hence $R_{k+1} \supseteq N[\{v_{0}, v_{1}\}] = \{v_{n-1}, v_{0}, v_{1}, v_{2}\}$.  If instead $n/2 \le i < n-2$, then the robber responds with $v_{i+1}$, which is consistent with both $v_{n-2}$ and $v_{n-1}$; thus $R_{k+1} \supseteq \{v_{n-3}, v_{n-2}, v_{n-1}, v_0\}$.  Similarly, if $i = n-2$ (respectively, $i=n-1$), then the robber responds with $v_{n-1}$ (resp. $v_0$); this response is consistent with both $v_0$ and $v_1$, so $R_{k+1} \supseteq \{v_{n-1}, v_{0}, v_{1}, v_{2}\}$.  Finally, if $i = 0$ the robber responds with $v_{n-1}$ and if $i = 1$ the robber responds with $v_0$; in either case the robber's response is consistent with both $v_{n-2}$ and $v_{n-1}$, hence $R_{k+1} \supseteq \{v_{n-3}, v_{n-2}, v_{n-1}, v_{0}\}$.  In any case, the robber has ensured that $R_{k+1}$ contains four consecutive vertices, as claimed.  It follows that $\dirloccomp(C_n) > 1$. 

Next, we argue that $\dirloc(C_n) \le 2$ by giving a strategy for two cops to locate a robber in the partial-feedback game on $C_n$.  As before, let the vertices of $C_n$ be labeled $v_0, v_1, \dots, v_{n-1}$ in order.  On the cops' first turn, the cops probe $v_{n-1}$ and $v_1$.  If the robber is on $v_{n-1}, v_0,$ or $v_1$, then this probe uniquely determines his location and the cops win.  Otherwise, after the ensuing recontamination phase, the robber could occupy any vertex in $\{v_1, v_2, \dots, v_{n-1}\}$.  On the cops' next turn, they probe $v_{n-1}$ and $v_2$.  If the robber is on $v_{n-1}$, $v_1$, or $v_2$ then the cops win; otherwise, after recontamination, the robber could occupy any vertex in $\{v_2, v_3, \dots, v_{n-1}\}$.  On the cops' next turn they probe $v_{n-1}$ and $v_3$, and so on.  Note that the robber territory shrinks by three vertices during each probing phase and grows by two vertices during each recontamination phase; the robber territory cannot keep shrinking forever, so eventually the cops must locate the robber.\\

\item [\textbf{(iv)}]  Let the partite sets of $K_{m,n}$ be $X$ and $Y$ with $\size{X} = m \ge n = \size{Y}$, and let $Y = \{y_1, y_2, \dots, y_n\}$.  We may suppose that $m \ge 2$, since otherwise the result follows from part (i).  In the full-feedback model of the game, it suffices to give a strategy for a single cop to locate the robber.  The cop begins by probing $y_1$.  Let $v$ denote the robber's vertex.  If $v \in X$, then the probe must return $v$ (and only $v$); if $v \in Y$, then every vertex of $X$ lies on a shortest $y_1,v$-path, so the probe must return all vertices in $X$.  Since $\size{X} \ge 2$, the cop can distinguish between these possibilities; hence, if $v \in X$, then the cop's probe uniquely determines the robber's location.  Otherwise, $v$ must be some vertex of $Y$ other than $y_1$.  Note that after the robber's move, he may occupy any vertex of $X$ or any vertex of $Y$ other than $y_1$.  Next, the cop probes $y_2$.  Once again, if the robber occupies a vertex in $X$, then the probe uniquely determines his location; otherwise, the robber must occupy some vertex of $Y$ other than $y_1$ or $y_2$.  In the next round of the game the cop probes $y_3$, and so on; eventually, the cop must locate the robber.  Consequently, $\dirloccomp(K_{m,n}) = 1$.  

For the partial-feedback model, we first argue that $\dirloc(K_{m,n}) \le n$.  With $n$ cops, the cops can simply probe every vertex of $Y$.  If the robber occupies a vertex in $Y$, then the cops have probed the robber's vertex and thereby win; if the robber is in $X$, then every probe must point to the robber's vertex, and again the cops win.  For the lower bound, we argue that $n-1$ cops do not suffice.  To do this, we argue that at the beginning of each round, the robber territory contains all vertices of the graph.  This is clearly true at the beginning of the first round.  Suppose, then, that it is true at the beginning of round $k$ for some $k$; we will argue that it is true at the beginning of round $k+1$ as well.  In round $k$, the cops probe $n-1$ vertices, so there must be some $x \in X$ and $y \in Y$ that do not get probed.  The robber returns $x$ to all probes in $Y$ and returns $y$ to all probes in $X$.  Note that both responses are consistent with both $x$ and $y$, so the cops have failed to determine the robber's location.  Moreover, immediately prior to recontamination, the robber territory contains a vertex from each partite set; hence, after recontamination, the robber territory contains all vertices of the graph, as claimed.  It follows that $n-1$ cops do not suffice to locate the robber, so $\dirloc(K_{m,n}) \ge n$. 
\end{enumerate}
\end{proof}

Finally, we show that one cop can win the directional localization game on any chordal graph.  Recall that a \textit{chordal} graph is one in which every cycle of length four or greater contains a chord, i.e. an edge between nonconsecutive vertices on the cycle.  In the following proof, we will actually use an alternative characterization of chordal graphs.  We call a vertex \textit{simplicial} if its neighborhood is a clique.  A \textit{simplicial elimination ordering} for a graph $G$ is an ordering $v_1, \dots, v_n$ of the vertices in $G$ such that each $v_i$ is simplicial in the subgraph of $G$ induced by $\{v_1, \dots, v_{i-1}\}$.  It is well-known that a graph is chordal if and only if it admits a simplicial elimination ordering. 

\begin{theorem}\label{thm:chordal}
If $G$ is a chordal graph, then $\dirloccomp(G) = \dirloc(G) = 1$.
\end{theorem}
\begin{proof}
By Proposition \ref{prop:complete_easier}, it suffices to show that $\dirloc(G) = 1$, so we give a strategy for one cop to win on $G$ in the partial-feedback model.  
We use induction on $\size{V(G)}$. The cop trivially wins if $\size{V(G)} = 1$.  Suppose otherwise, and let $v$ be the last vertex in a simplicial elimination ordering for $G$.  Note that $G-v$ itself has a simplicial elimination ordering and is thus chordal; hence, one cop can win the partial-feedback game on $G-v$.  We will show that a single cop can win on $G$ as well.  The cop will use a winning strategy on $G-v$ to guide her play on $G$.  To this end, we consider two games: the ``real'' game on $G$, and an ``imagined'' game on $G-v$. Before explaining the cop's strategy, we make a few observations.  
\begin{itemize}
\item First, note that for any vertices $a$ and $b$ in $G-v$, there is no shortest $a,b$-path in $G$ that uses $v$.  This is because in any such path, $v$ must be preceded by some neighbor $u$ and followed by some neighbor $w$; since $v$ is simplicial in $G$, vertices $u$ and $w$ must be adjacent, so traveling directly from $u$ to $w$ without visiting $v$ would yield a shorter $a,b$-path.

\medskip

\item Next, suppose that in both games, the cop probes some vertex $w$ and the robber responds with $x$.  We claim that a vertex $z$ in $V(G) - \{v\}$ is consistent with the probe in $G$ if and only if it is consistent with the probe in $G-v$.  This is an immediate consequence of the preceding observation: the shortest paths $w,z$-paths in $G$ are the same as those in $G-v$, and $z$ is consistent with the probe if and only if some such path includes $x$.

\medskip

\item Finally, note that in the game on $G$, if $v$ is contaminated at the beginning of some round, then all vertices of $N(v)$ must also be contaminated.  This is because $v$ is contaminated at the beginning of round $k$ if and only if either: $k=1$ (in which case all vertices are contaminated); $v$ was contaminated prior to the recontamination phase of round $k-1$ (in which case $v$ spread its contamination to all of $N(v)$); or during the recontamination phase of round $k-1$, contamination spread to $v$ from some neighbor $u$ (in which case contamination also spread from $u$ to all of $N(v)$, because $N(v) \subseteq N[u]$ due to $v$ being simplicial).
\end{itemize}

We are now ready to give the cop's winning strategy on $G$.  As long as the imagined game on $G-v$ has not yet ended, the cop plays identically in both games.  More precisely, suppose that in some round of the game, the cop's winning strategy on $G-v$ tells her to probe a vertex $w$.  The cop does so, and then probes $w$ in the real game on $G$ as well.  Suppose that the robber responds, in the real game, with some vertex $x$.  If $x = v$, then the only vertex consistent with the probe is $v$ itself, so the cop wins the real game.  Otherwise, the cop imagines that the robber responded with $x$ in the game on $G-v$.  Play continues in this manner unless and until the cop wins the game on $G-v$; we will explain later how the cop plays past that point.

Before that, we claim that at all points during the game, so long as the cop has not yet won either game, a vertex $z$ in $V(G)-\{v\}$ is contaminated in the game on $G-v$ if and only if it is also contaminated in the game on $G$.  This is clearly true at the beginning of the first round.  Fix some positive integer $k$, and assume that the claim is true at the beginning of round $k$; we will show that it is also true after the probing phase of round $k$ and at the beginning of round $k+1$.  Suppose that during round $k$, the cop probes vertex $w$ and the robber responds with $x$.  Fix any vertex $z$ in $V(G) - \{v\}$.  As argued above, $z$ is consistent with the probe in $G$ if and only if it is also consistent with the probe in $G-v$; hence, after the probing phase, $z$ is contaminated in $G$ if and only if it is contaminated in $G-v$.  Now consider the recontamination phase.  Recontamination happens identically in both graphs except that in $G$, contamination could potentially spread from $v$ to its neighbors.  However, if $v$ is contaminated in $G$, then it must have been consistent with the probe.  We know that $x \not = v$ (since otherwise the real game would be over), so any shortest path from $w$ to $v$ must pass through some neighbor $u$ of $v$, causing $u$ to be consistent with the probe as well.  As argued above, if $v$ was contaminated at the beginning of round $k$, then so were all neighbors of $v$, in particular $u$.  Thus $u$ was both contaminated and consistent with the probe in $G$ and thus also in $G-v$, and so $u$ will spread contamination in $G-v$ to all vertices of $N_G(v)$.  It follows that the claim holds at the beginning of round $k+1$, as desired.

Since the cop follows a winning strategy for the game on $G-v$, eventually she determines the robber's location in that game.  Suppose that after the probing phase of round $k$, the cop determines that the robber occupies vertex $z$ in $G-v$.  Recall that a vertex of $G-v$ is contaminated in the game on $G-v$ if and only if it is also contaminated in the game on $G$.  If $z$ is the only contaminated vertex in the game on $G$, then the cop wins that game as well.  The only other possibility is that both $z$ and $v$ are contaminated in the game on $G$.  This can only happen if $z \in N(v)$: since $v$ is contaminated, it must be that all of $N[v]$ was contaminated prior to the probing phase and that $z$ was consistent with the probe; as argued in the previous paragraph, it follows that some neighbor of $z$ was also consistent with the probe and is thus also contaminated after the probing phase.  In this case, after the recontamination phase of round $k$, the robber territory will be $N[\{v,z\}]$; however, since $z \in N(v)$ and $v$ is simplicial, we have $N[\{v,z\}] = N[z]$.  
Thus, the cop can win the game on $G$ by probing $z$ in round $k+1$; whichever vertex the probe returns must be the robber's location.
\end{proof}

\section{General Bounds}\label{sec:dirloc_genbounds}

In this section, we give general bounds on $\dirloc(G)$ and $\dirloccomp(G)$ based on various structural properties of $G$.  Our focus is on upper bounds for $\dirloc(G)$; however, note that by Proposition \ref{prop:complete_easier}, any upper bound on $\dirloc(G)$ applies also to $\dirloccomp(G)$.

We begin by exploring the connection between the partial-feedback directional localization number of a graph and its \emph{degeneracy}.  Recall that a graph $G$ is \emph{$k$-degenerate} if $\delta(H) \le k$ for every subgraph $H$ of $G$; the \emph{degeneracy} of $G$ is the maximum $k$ such that $G$ is $k$-degenerate.

\begin{theorem}\label{thm:degen_upper}
If $G$ is $k$-degenerate, then $\dirloc(G) \le k+1$.
\end{theorem}
\begin{proof}
Let $G$ be $k$-degenerate, and let $v_1, v_2, \dots, v_n$ denote the vertices of $G$, ordered so that each $v_i$ has degree at most $k$ in the subgraph of $G$ induced by $\{v_1, v_2, \dots, v_i\}$.  (This is possible because $G$ is $k$-degenerate; one can construct such an ordering by choosing each $v_i$, from $v_n$ down to $v_1$, to be a vertex of minimum degree in the subgraph induced by $\{v_1, \dots, v_i\}$.) 

We give a strategy for $k+1$ cops to locate a robber on $G$ in the partial-feedback game.  Throughout the game, the cops will distinguish one vertex of $G$ as their ``base''.  Initially, the cops choose an arbitrary vertex $v_i$ to be the base. In the first round of the game, the cops probe every neighbor $v_j$ of $v_i$ such that $j < i$, along with $v_i$ itself (which is possible since $v_i$ has at most $k$ neighbors $v_j$ such that $j < i$).  Let $d$ denote the distance from $v_i$ to the robber's position.  Note that the cops do not know $d$, and as such, they cannot use it to inform their strategy; however, it will be helpful to refer to $d$ in our analysis of the cops' strategy.

If $d = 0$, then the cops have probed the robber's vertex and thereby won, so suppose $d \ge 1$.  We consider two cases.  
\begin{itemize}
\item \textbf{Case 1:} the probe at $v_i$ returns some neighbor $v_j$ with $j < i$.  In this case, the cops have probed $v_j$.  If $d=1$, then the robber must occupy $v_j$; hence the cops have probed the robber's vertex and won.  Suppose instead that $d \ge 2$.  Let $w$ denote the vertex returned by the probe at $v_j$, and note that the distance from $w$ to the robber must be $d-2$.  Hence, after the robber moves, the distance from $w$ to the robber will be at most $d-1$.  The cops now take $w$ to be their new base vertex and repeat their strategy.  Note that in this case, the distance from the base to the robber decreases.

\medskip

\item \textbf{Case 2:} the probe at $v_i$ returns some neighbor $v_j$ with $j > i$.  In this case, the cops have not probed $v_j$.  However, the distance from the robber to $v_j$ is $d-1$, so after the robber's move, it will be at most $d$.  The cops now move their base to $v_j$ and repeat their strategy.  Note that in this case, the distance from the base to the robber need not decrease; however, it cannot increase.  Moreover, the cops' base necessarily moves from $v_i$ to $v_j$ with $j > i$ -- that is, it moves ``forward'' in the sequence $v_1, v_2, \dots, v_n$.  
\end{itemize}
Each time the cops enter Case 1, the distance from the base to the robber decreases, and each time they enter Case 2, the distance from the base to the robber does not increase.  Hence, if the cops enter Case 1 often enough, then they necessarily locate the robber.  Moreover, each time the cops enter Case 2, the base moves ``forward'' in the sequence $v_1, v_2, \dots, v_n$.  Hence the cops cannot enter Case 2 more than $n-1$ times consecutively, so they must enter Case 1 at least once every $n$ rounds; it follows that they eventually enter Case 1 enough times to win the game.
\end{proof}

The following corollaries follow immediately from Theorem \ref{thm:degen_upper}.
\begin{corollary}\label{cor:outerplanar_degen}
If $G$ is outerplanar, then $\dirloc(G) \le 3$.
\end{corollary}
\begin{corollary}\label{cor:planar_degen}
If $G$ is planar, then $\dirloc(G) \le 6$.
\end{corollary}
\begin{corollary}\label{cor:treewidth_degen}
For every graph $G$, we have $\dirloc(G) \le \tw(G)+1$, where $\tw(G)$ denotes the treewidth of $G$.
\end{corollary}

Figure \ref{fig:2_degen_tight} shows a graph $G$ such that $G$ has degeneracy 2, but $\dirloc(G) = 3$; this shows that Theorem \ref{thm:degen_upper} is tight when $k=2$.  
However, Theorem \ref{thm:chordal} shows that Theorem \ref{thm:degen_upper} is \textit{not} tight when $k=1$, since every 1-degenerate graph is a tree and thus chordal.  (It is an open question as to whether or not Theorem \ref{thm:degen_upper} is tight when $k \ge 3$.) 


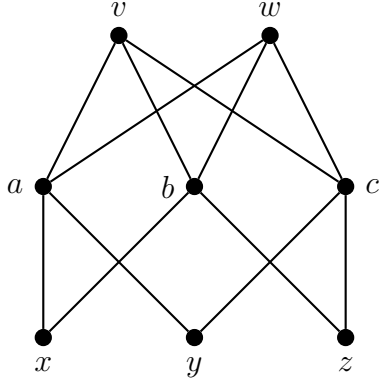
\begin{figure}
\centering
\begin{tikzpicture}[inner sep=0mm, thick,
 smallvertex/.style={draw=black, circle, minimum size=0.011cm},
 vertex/.style={draw=black, fill=black, circle, minimum size=0.2cm},
 xscale=1,yscale=2]

\node (v) at (1,0) [vertex] {};
\node at (v) [above = 0.25cm] {$v$};
\node (w) at (3,0) [vertex] {};
\node at (w) [above = 0.25cm] {$w$};
\node (a) at (0,-1) [vertex] {};
\node at (a) [left = 0.25cm] {$a$};
\node (b) at (2,-1) [vertex] {};
\node at (b) [left = 0.25cm] {$b$};
\node (c) at (4,-1) [vertex] {};
\node at (c) [right = 0.25cm] {$c$};
\node (x) at (0,-2) [vertex] {};
\node at (x) [below = 0.25cm] {$x$};
\node (y) at (2,-2) [vertex] {}; 
\node at (y) [below = 0.25cm] {$y$};
\node (z) at (4,-2) [vertex] {};
\node at (z) [below = 0.25cm] {$z$};

\draw (v) -- (a) -- (w);
\draw (v) -- (b) -- (w);
\draw (v) -- (c) -- (w);
\draw (x) -- (a) -- (y);
\draw (x) -- (b) -- (z);
\draw (y) -- (c) -- (z);
\end{tikzpicture}
\caption{A 2-degenerate graph $G$ with $\dirloc(G) = 3$.}
\label{fig:2_degen_tight}
\end{figure}

Theorem \ref{thm:degen_upper} implies that for all graphs $G$, we have $\dirloc(G) \le \Delta(G)+1$.  For graphs with high maximum degree, a simpler argument yields a stronger upper bound.

\begin{proposition}\label{prop:max_degree_dense}
For every $n$-vertex graph $G$, we have $\dirloc(G) \le n-\Delta(G)$.
\end{proposition}
\begin{proof}
We give a strategy for $n-\Delta(G)$ cops to locate a robber on $G$.  Let $v$ be a vertex of maximum degree in $G$.  On the first turn of the game, one cop probes $v$, while the other $n-\Delta(G)-1$ cops probe all vertices not belonging to $N[v]$.  If any of the cops probes the robber's vertex, then the cops win; otherwise, the robber must occupy some vertex in $N(v)$, so the probe at $v$ uniquely determines the robber's location.
\end{proof}

\begin{corollary}\label{cor:n_over_two}
For every $n$-vertex graph $G$, we have $\dirloc(G) \le \floor{n/2}$.
\end{corollary}
\begin{proof}
By integrality of $\dirloc(G)$, it suffices to show that $\dirloc(G) \le n/2$.  If $\Delta(G) \ge n/2$, then by Proposition \ref{prop:max_degree_dense} we have $\dirloc(G) \le n-\Delta(G) \le n-n/2 = n/2$.  If instead $\Delta(G) \le n/2-1$, then by Theorem \ref{thm:degen_upper} we have $\dirloc(G) \le \Delta(G)+1 \le n/2$.  

This covers all cases except when $n = 2k+1$ for some $k$ and $\Delta(G) = k$.  For this case, we use a modification of the cop strategy in Proposition \ref{prop:max_degree_dense}.  Let $v$ be a vertex of maximum degree.  On their first turn, the cops probe every vertex \textit{except} for those in $N(v)$ and an arbitrary vertex $x$ at distance 2 from $v$.  If the cops have probed the robber's vertex, then they win.  Otherwise, suppose that the probe at $v$ returns $w$.  If $w$ is not adjacent to $x$, then the robber must occupy $w$ and again the cops win.  Suppose instead that $w$ is adjacent to $x$; the cops now know that the robber must occupy either $w$ or $x$.  After the recontamination phase, the robber could occupy any vertex in $N[w] \cup N[x]$.  Note that since $w$ and $x$ are adjacent, we have $N[w] \cup N[x] = N(w) \cup N(x)$.  On the second cop turn, the cops probe every vertex in $N(x)$, which is possible since $\size{N(x)} \le \Delta(G) = k$.  If the cops have probed the robber's vertex, then they win; otherwise, the robber must occupy some vertex in $N(w)$.  However, since $w \in N(x)$, the cops have probed $w$, so the probe at $w$ must return the robber's location.
\end{proof}

Perhaps surprisingly, Corollary \ref{cor:n_over_two} is tight for all values of $n$.  When $n = 2k$ for some $k$, by Proposition \ref{prop:simple}(iv) we have $\dirloc(K_{k,k}) = k = n/2$; similarly, when $n=2k+1$ we have $\dirloc(K_{k,k+1}) = k = \floor{n/2}$.  

We next turn our attention to the behavior of the game on Cartesian products of graphs.  Recall that the \textit{Cartesian product} of graphs $G$ and $H$, denoted $G \cart H$, is the graph with vertex set $V(G) \times V(H)$, with edges joining $(u,v)$ with $(u',v')$ whenever $uu' \in E(G)$ and $v=v'$, or $vv' \in E(H)$ and $u=u'$.  The graphs $G$ and $H$ are referred to as \textit{factors} of $G \times H$.  Given a vertex $(u,v) \in V(G \cart H)$, we refer to $u$ and $v$ as the \textit{$G$-coordinate} and \textit{$H$-coordinate} of the vertex, respectively.  For a set $S$ of vertices in $G \cart H$, the \textit{projection of $S$ onto $G$} is the set $\{u \in V(G) \, \vert \, (u,v) \in S \text{ for some } v \in V(H)\}$; likewise, the \textit{projection of $S$ onto $H$} is the set $\{v \in V(H) \, \vert \, (u,v) \in S \text{ for some } u \in V(G)\}$.


In the full-feedback model, the directional localization number of $G \cart H$ is easily determined from the directional localization numbers of its factors.

\begin{theorem}\label{thm:cartesian_complete}
For any graphs $G$ and $H$, we have $\dirloccomp(G \cart H) = \max\{\dirloccomp(G), \dirloccomp(H)\}$.
\end{theorem}
\begin{proof}
Assume without loss of generality that $\dirloccomp(G) \ge \dirloccomp(H)$ and let $k = \dirloccomp(G)$; we will argue that $\dirloccomp(G \cart H) \ge k$ and that $\dirloccomp(G \cart H) \le k$.

The lower bound is clear: if the robber has a strategy to evade $k$ cops on $G$, then on $G \cart H$, he can simply ignore $H$-coordinates and employ his winning strategy on $G$ to prevent $k$ cops from determining the $G$-coordinate of his position.

For the upper bound, we give a strategy for $k$ cops to locate a robber on $G \cart H$.  Essentially, the cops will follow a winning strategy on $G$ and a winning strategy on $H$ simultaneously.  Formally, the cops play as follows.  The cops imagine games on $G$ and $H$, and they use these games to guide their play on $G \cart H$.  In each round, the cops choose probes $u_1, \dots, u_k$ in $G$ and $v_1, \dots, v_k$ in $H$ according to winning strategies for those games; in $G \cart H$, they probe $(u_1,v_1), \dots, (u_k,v_k)$.  For each $i$, the response to cop $i$'s probe in $G \cart H$ will necessarily be of the form
\[\{(u_i,x) \, \vert \, x \in S^{(i)}_H\} \cup \{(w,v_i) \, \vert \, w \in S^{(i)}_G\}\]
for some set $S^{(i)}_G$ of vertices of $G$ and some set $S^{(i)}_H$ of vertices of $H$.  In the imagined game on $G$, the cops imagine that the robber has responded with $S^{(i)}_G$, while in the imagined game on $H$, the cops imagine that he has responded with $S^{(i)}_H$.  They then update the imagined games and repeat the process.  

Once the cops locate the robber in one of the imagined games, they play slightly differently: instead of ending the imagined game, they let the game continue, and they keep locating the robber in each successive round.  (Note that if the cops determine that the robber is located on vertex $v$, then after the subsequent recontamination phase, the robber must be in $N[v]$; thus, by probing $v$ in the next round, the cops can locate the robber once again.  Hence, once they have located the robber once, they can continue to locate him indefinitely.)  The cops repeat their strategy until either they have located the robber on $G \cart H$ or they have located him in both imagined games simultaneously.


We claim that at all points during the game, if the robber territories in the imagined games on $G$ and $H$ are $R_G$ and $R_H$ respectively, then the robber territory in the actual game on $G \cart H$ must be a subset of $R_G \times R_H$.  As usual, for $r \ge 1$, let $R_{G,r}$ and $R'_{G,r}$ (resp. $R_{H,r}$ and $R'_{H,r}$) denote the robber territories in $G$ (resp. $H$) just before and just after the probing phase of round $r$.  Likewise, let $R_r$ and $R'_r$ denote the robber territories in $G \cart H$ just before and just after the probing phase of round $r$.  It is clear that $R_1 \subseteq R_{G,1} \times R_{H,1}$. Fix $r \ge 1$, and assume that $R_r \subseteq R_{G,r} \times R_{H,r}$; we will show that $R'_{r} \subseteq R'_{G,r} \times R'_{H,r}$ and that $R_{r+1} \subseteq R_{G,r+1} \times R_{H,r+1}$.  

For all $i \in \{1, \dots, k\}$, let $(u_i,v_i)$ be the vertex probed by cop $i$ in round $r$, let $S^{(i)}$ be the robber's response, let $S^{(i)}_G = \{u \, \vert \, (u,v_i) \in S^{(i)}\}$, and let $S^{(i)}_H = \{v \, \vert \, (u_i,v) \in S^{(i)}\}$.  For any $(w,x)$ in $R'_{r}$, we must have $(w,x) \in R_{r}$ and, moreover, for all $(u,v)$ in $S^{(i)}$, some shortest path from $(u_i,v_i)$ to $(w,x)$ must pass through $(u,v)$.  Note that every $(u,v)$ in $S^{(i)}$ either has the form $(u,v_i)$ for some $u \in S_G^{(i)}$ or $(u_i,v)$ for some $v \in S_H^{(i)}$.  A shortest path from $(u_i,v_i)$ to $(w,x)$ passes through $(u,v_i)$ if and only if some shortest $u_i,w$-path in $G$ passes through $u$; likewise, a shortest path from $(u_i,v_i)$ to $(w,x)$ passes through $(u_i,v)$ if and only if some shortest $v_i,x$-path in $H$ passes through $v$.  Consequently, $(w,x) \in R'_{r}$ implies the following:
\begin{itemize}
\item $(w,x) \in R_{r} \subseteq R_{G,r} \times R_{H,r}$, hence $w \in R_{G,r}$ and $x \in R_{H,r}$;
\item For all $u \in S_G^{(i)}$, some shortest $u_i,w$-path in $G$ passes through $u$; and
\item For all $v \in S_H^{(i)}$, some shortest $v_i,x$-path in $G$ passes through $v$.
\end{itemize}
It follows that $w \in R'_{G,r}$ and $x \in R'_{H,r}$.  Thus, 
\[R'_{r} \subseteq R'_{G,r} \times R'_{H,r}\]
and, moreover,
\[R_{r+1} = N[R'_{r}] \subseteq N[R'_{G,r} \times R'_{H,r}] \subseteq N[R'_{G,r}] \times N[R'_{H,r}] = R_{G,r+1} \times R_{H,r+1}\]
as claimed.


Because the cops follow winning strategies in the imagined games on $G$ and $H$, eventually either the game on $G \cart H$ ends, or the cops locate the robber on both $G$ and $H$.  Suppose the latter, and suppose that this happens in the probing phase of round $r$ for some $r$.  Because the cops have determined the robber's position in both imagined games, we have $R'_{G,r} = \{u\}$ and $R'_{H,r} = \{v\}$ for some $u$ and $v$; hence, in the game on $G \cart H$, we have $R'_{r} \subseteq \{u\} \times \{v\} = \{(u,v)\}$, so the robber must be located on vertex $(u,v)$.  Thus the cops have in fact won on $G \cart H$ as well.
\end{proof}

For the partial-feedback model, things are less clear-cut.  Since a graph's degeneracy is bounded above by its maximum degree, Theorem \ref{thm:degen_upper} implies that always $\dirloc(G \cart H) \le \Delta(G \cart H) + 1 = \Delta(G) + \Delta(H) + 1$.  We next establish a lower bound on $\dirloc(G \cart H)$ that, in some circumstances, nearly matches this upper bound.

\begin{theorem}\label{thm:cartesian_lower}
For any nonempty connected graphs $G$ and $H$ we have $\dirloc(G \cart H) \ge \max\{\delta(G) + \dirloc(H), \delta(H) + \dirloc(G)\}$.
\end{theorem}
\begin{proof}
Suppose without loss of generality that $\delta(G) + \dirloc(H) \ge \delta(H) + \dirloc(G)$, and let $k = \delta(G) + \dirloc(H)$; we give a strategy for the robber to evade $k-1$ cops.  Very loosely, the robber's strategy will be to respond to every probe with information about the $G$-coordinate of his current position whenever possible.  To simplify the argument, we will actually consider a modified game on $G \cart H$, wherein the robber informs the cops that he will choose a starting location in $N_G[x] \times V(H)$ for some particular (but arbitrarily-chosen) $x \in V(G)$.  If the robber can evade $k-1$ cops in this modified game, then clearly he can do so in the original game as well.  
(Although this tactic may seem detrimental to the robber, it won't actually harm him: since the robber plans to give the cops information about the $G$-coordinate of his position at every opportunity, the cops would be able to determine the $G$-coordinate of the his position without much difficulty.)

The robber imagines a game on $H$ played against $\dirloc(H)-1$ cops, and he uses a winning strategy in that game to guide his play on $G \cart H$.  As usual, for all $r \ge 1$, let $R_{H,r}$ and $R'_{H,r}$ denote the robber territory in the imagined game on $H$ just before and just after the probing phase of round $r$.  Likewise, let $R_r$ and $R'_r$ denote the robber territory in the actual game on $G \cart H$ just before and just after the probing phase of round $r$.  The robber will play so as to ensure that for all $r$, we have $R_r \supseteq \left(\{x_r\} \times R_{H,s}\right) \cup \left(N_G[x_r] \times R'_{H,s-1}\right )$ and $R'_r \supseteq \{x'_r\} \times R'_{H,s}$ for some positive integer $s$ with $s \le r$ and some $x_r,x'_r \in V(G)$.  

By assumption we have $R_1 \supseteq \left(\{x\} \times R_{H,1}\right) \cup \left (N_G[x] \times R'_{H,0}\right)$.  Fix $r \ge 1$ and assume that $R_r \supseteq \left(\{x_r\} \times R_{H,s}\right) \cup \left(N_G[x_r] \times R'_{H,s-1}\right )$ for some $s \le r$ and some $x_r \in V(G)$; we will argue that $R'_r \supseteq \{x_{r+1}\} \times R'_{H,s-1}$ and $R_{r+1} \supseteq \left(\{x_{r+1}\} \times R_{H,s'}\right) \cup \left(N_G[x_{r+1}] \times R'_{H,s'-1}\right )$ for some $s' \le r+1$ and some $x_{r+1} \in V(G)$.  We may assume that in fact $R_r = \left(\{x_r\} \times R_{H,s}\right) \cup \left(N_G[x_r] \times R'_{H,s-1}\right )$, since reducing the robber territory cannot benefit the robber.   Suppose that in round $r$ of the game on $G \cart H$, the cops probe vertices $(u_1, v_1), \dots, (u_{k-1},v_{k-1})$.  We consider two cases.
\begin{itemize}
\item \textbf{Case 1:} there is some $y \in N_G[x_r]$ such that $y \not \in \{u_1, \dots, u_{k-1}\}$.  In the game on $G \cart H$, for all $i \in \{1, \dots, k-1\}$, the robber responds to the probe $(u_i,v_i)$ with vertex $(u,v_i)$, where $u$ is a neighbor of $u_i$ in the direction of $y$.  Note that all vertices in $\{y\} \times R'_{H,s-1}$ are consistent with the robber's response.  Additionally, all such vertices are contaminated, since by assumption $R_r \supseteq N_G[x_r] \times R'_{H,s-1}$.  Now taking $x_{r+1} = y$ and $s' = s$, we have $R'_r \supseteq \{x_{r+1}\} \times R'_{H,s'-1}$ and, moreover,
\begin{align*}
R_{r+1} &= N[R'_r]\\
    &\supseteq N[\{x_{r+1}\} \times R'_{H,s'-1}]\\
    &\supseteq \left(N_G[x_{r+1}] \times R'_{H,s-1}\right) \cup \left(\{x_{r+1}\} \times N[R'_{H,s'-1}]\right)\\
    &= \left(N_G[x_{r+1}] \times R'_{H,s'-1}\right) \cup \left(\{x_{r+1}\} \times R_{H,s'}\right),
\end{align*}
as desired.\\

\item \textbf{Case 2:} every vertex in $N_G[x_r]$ appears in $\{u_1, \dots, u_{k-1}\}$.  Let $c$ denote the number of different $i \in \{1, \dots, k-1\}$ such that $u_i = x_r$.  Note that since every neighbor of $x_r$ appears in $\{u_1, \dots, u_{k-1}\}$, we have $c \le k-1-\deg(x_r) \le k-1-\delta(G) = \dirloc(H)-1$.  

Suppose without loss of generality that $u_1 = u_2 = \dots = u_c = x_r$ and that $u_i \not = x_r$ for $i > c$.  For all $i \in \{c+1, \dots, k-1\}$, in the game on $G \cart H$, the robber responds to the probe at $(u_i,v_i)$ with vertex $(u,v_i)$, where $u$ is any neighbor of $u_i$ in the direction of $x_r$.  Note that these responses are all consistent with all vertices in $\left(\{x_r\} \times R_{H,s}\right) \cup \left(\{x_r\} \times R'_{H,s-1}\right)$.  Next, in the game on $H$, the robber imagines that the cops probe vertices $v_1, \dots, v_c$.  If in fact $c < \dirloc(H)-1$, then the robber additionally imagines $\dirloc(H)-1-c$ additional probes, at arbitrary vertices.  Since the robber is playing the imagined game on $H$ against fewer than $\dirloc(H)$ cops, he has a winning strategy; suppose that he responds to these probes according to such a strategy, and let his responses to $v_1, \dots, v_c$ be $y_1, \dots, y_c$.  In the actual game, for all $i \in \{1, \dots, c\}$, the robber responds to the probe at $(u_i,v_i)$ with vertex $(u_i,y_i)$.

We aim to determine the set of vertices in $R'_r$ of the form $(x_r, y)$.  As noted above, every vertex in $\{x_r\} \times R_{H,s}$ is consistent with the robber's responses to the probes at $(u_{c+1},v_{c+1}), \dots, (u_{k-1},v_{k-1})$.  A vertex $(x_r, y)$ is consistent with the robber's responses to probes at $(x_r,v_1), \dots, (x_r,v_c)$ provided that for all $i \in \{1, \dots, c\}$, some shortest $v_i,y$-path in $H$ passes through $y_i$.  By definition of $R'_{H,s}$, this is true for all $y \in R'_{H,s}$.  Hence, taking $x_{r+1} = x_r$, we have
\[R'_r \supseteq \{x_{r+1}\} \times R'_{H,s}\]
and
\begin{align*}
R_{r+1} &= N[R'_r]\\ 
    &\supseteq N[\{x_{r+1}\} \times R'_{H,s}]\\
    &\supseteq \left (\{x_{r+1}\} \times N_H[R'_{H,s}]\right ) \cup \left(N_G[x_{r+1}] \times R'_{H,s}\right)\\
    &= \left (\{x_{r+1}\} \times R_{H,s+1}\right ) \cup \left(N_G[x_{r+1}] \times R'_{H,s}\right),
\end{align*}
as desired.
\end{itemize} 

To show that this strategy allows the robber to evade the cops indefinitely, it suffices to argue that $\size{R'_r} \ge 2$ for all $r$.  In the imagined game on $H$, the robber uses a winning strategy against $\dirloc(H)-1$ cops; hence, $\size{R'_{H,s}} \ge 2$ for all $s$.  It follows that $\size{R'_r} \ge \size{\{x'_r\} \times R'_{H,s}} \ge 2$ as well; thus, the cops can never uniquely determine the robber's position, as claimed. 
\end{proof}

Together, Theorem \ref{thm:degen_upper} and Theorem \ref{thm:cartesian_lower} nearly determine $\dirloc(Q_n)$.

\begin{cor}\label{cor:hypercube}
For every positive integer $n$, we have $\dirloc(Q_n) \in \{n, n+1\}$.
\end{cor}
\begin{proof}
By Theorem \ref{thm:degen_upper}, we have $\dirloc(Q_n) \le \Delta(Q_n)+1 = n+1$.  For the lower bound, we use induction on $n$. Clearly, $\dirloc(Q_1) = \dirloc(K_2) = 1$.  Fix $n \ge 2$, and assume that $\dirloc(Q_{n-1}) \ge n-1$; now by Theorem \ref{thm:cartesian_lower}, 
\[\dirloc(Q_n) = \dirloc(Q_{n-1} \cart K_2) \ge \dirloc(Q_{n-1}) + \Delta(K_2) \ge (n-1)+1 = n,\]
as claimed.
\end{proof}

\section{Treewidth}\label{sec:dirloc_treewidth}

In this section, we explore the connection between the partial-feedback directional localization number of a graph and its treewidth.  Corollary \ref{cor:treewidth_degen} shows that for any graph $G$, the partial-feedback directional localization number of $G$ is at most $\tw(G)+1$.  In this section, we refine this bound by showing that under certain circumstances, we have $\dirloc(G) \le tw(G)$.

We begin with a formal definition of treewidth.
\begin{defn}\label{defn:treewidth}
A \textit{tree decomposition} of a graph $G$ is a pair $(T, \mathcal{B})$, where $T$ is a tree and $\mathcal{B} = \{B_a \, : \, a \in V(T)\}$ is a family of subsets of $V(G)$, with the following two properties:
\begin{itemize}
\item [(i)] For any $v \in V(G)$ and $B_a, B_b \in V(T)$, if $v$ is an element of both $B_a$ and $B_b$, then it is also an element of $B_c$ for every node $c$ on the unique $a,b$-path in $T$; and
\item [(ii)] For every $uv \in E(G)$, some $B_a$ contains both $u$ and $v$.
\end{itemize}

We refer to the sets $B_a$ as \textit{bags} of the decomposition.  The \textit{width} of a decomposition is one less than the maximum size of a bag; the \textit{treewidth} of $G$, denoted $\tw(G)$, is the minimum width over all tree decompositions of $G$.  
\end{defn}

In what follows, we will need our tree decompositions to have certain nice properties.

\begin{defn}[\cite{Bod96}]\label{defn:smooth}
We say that a tree decomposition $(T, \mathcal{B})$ of width $k$ is \textit{smooth} if:
\begin{itemize}
\item $\size{B_a} = k+1$ for all $a \in V(T)$, and
\item $\size{B_a \cap B_b} = k$ for all $ab \in E(T)$.
\end{itemize}
\end{defn}

It is well-known (see e.g. \cite{Bod96}) that if $G$ has a tree decomposition of width $k$, then it has a smooth tree decomposition of width $k$.  For our purposes, we will be interested in an even more specialized sort of tree decomposition, which we refer to as a \textit{restricted smooth tree decomposition}.

\begin{defn}\label{defn:restricted}
Given a tree decomposition $(T, \mathcal{B})$ of width $k$, we call $(T, \mathcal{B})$ a \textit{restricted smooth tree decomposition} of $G$ if:
\begin{itemize}
\item $\Delta(T) \le 3$;
\item $\size{B_a} \in \{k, k+1\}$ for all $a \in V(T)$; and
\item $\size{B_a \cap B_b} = k$ for all $ab \in E(T)$.
\end{itemize}
\end{defn}

Given any graph $G$, a smooth tree decomposition of $G$ satisfies the second and third properties of \ref{defn:restricted}.  Additionally, it is not difficult to find a decomposition of width $\tw(G)$ in which every node of the underlying tree has degree at most 3.  However, it is not always possible to find a decomposition that accomplishes both of these feats simultaneously; that is, not every graph $G$ admits a restricted smooth tree decomposition of width $\tw(G)$.  

We will also need the following well-known fact.

\begin{lemma}[\cite{Diestel}, Lemma 12.3.1]\label{lem:bag_cutset}
Let $(T, \mathcal{B})$ be a tree decomposition of a graph $G$, let $a$ and $b$ be adjacent nodes in $T$, and let $T_a$ (resp. $T_b$) be the component of $T-ab$ containing $a$ (resp. $b$).  Then $B_a \cap B_b$ separates $\cup_{c \in V(T_a)} B_c$ from $\cup_{c \in V(T_b)} B_c$ in $G$. 
\end{lemma}


Very loosely, we can use a tree decomposition $(T,\mathcal{B})$ of $G$ to locate a robber in $G$ by probing all vertices in the intersection of bags $B_a$ and $B_b$ for some $ab \in E(T)$, determining which component of $T-ab$ contains the robber, and repeating.  Our next lemma shows that, when probing $B_a \cap B_b$, it is in fact possible for the cops to determine which component of $T-ab$ contains the robber.

\begin{lemma}\label{lem:oneside}
Let $G$ be a graph, and let $(T, \mathcal{B})$ be a 
tree decomposition.  Let $a$ and $b$ be adjacent nodes in $T$ and let $S = B_a \cap B_b$.  If the cops probe every vertex of $S$, then regardless of how the robber responds to these probes, all vertices consistent with the robber's response lie in bags corresponding to nodes in the same component of $T-ab$.
\end{lemma}
\begin{proof}
Let $S = \{v_1, \dots, v_k\}$.  By Lemma \ref{lem:bag_cutset}, $S$ is a cut-set in $G$; let $H$ (resp. $H'$) be the subgraph of $G-S$ consisting of all vertices in bags corresponding to nodes on the same side as $a$ (resp. $b$) in $T-ab$.  Suppose the cops probe all vertices of $S$ and the robber responds.  
We aim to show that either all vertices consistent with the robber's response lie in $H$, or all such vertices lie in $H'$.  Suppose for the sake of contradiction that there exist vertices $r$ and $r'$, both consistent with the robber's response, such that $r \in V(H)\setminus V(H')$ and $r' \in V(H') \setminus V(H)$.  

Because $r$ and $r'$ lie in different components of $G-S$, every $r,r'$-path must include at least one vertex in $S$.  Let $Q$ denote a shortest $r,r'$-path, and suppose without loss of generality that $Q$ includes $v_1$.  Because $v_1$ lies on a shortest $r,r'$-path, we have $\dist(r, r') = \dist(r,v_1) + \dist(v_1, r')$.  Let $w_1$ denote the robber's response to the cops' probe at $v_1$.  Both $r$ and $r'$ are consistent with the probe at $v_1$, hence there must exist an $r,v_1$-path $Q_1$ that has length $\dist(r, v_1)$, and whose next-to-last vertex is $w_1$; likewise, there must be some $v_1,r'$-path $Q_1'$ that has length $\dist(v_1, r')$ and whose second vertex is $w_1$.  Now consider the $r,r'$-path that follows $Q_1$ from $r$ to $w_1$, then follows the portion of $Q_1'$ that runs from $w_1$ to $r'$.  This path uses all vertices of $Q_1$ except the last, and all vertices of $Q_1'$ except the first, so it has length 
$$(\dist(r, v_1) - 1) + (\dist(v_1, r')-1) = \dist(r,v_1)+\dist(v_1,r')-2 = \dist(r,r')-2,$$
which is clearly impossible.  The claim now follows.
\end{proof}

We are now ready to establish a bound on $\dirloc(G)$ in terms of the width of a restricted smooth tree decomposition of $G$.

\begin{theorem}\label{thm:treewidth_restricted}
If a graph $G$ has a restricted smooth tree decomposition of width $k$, then $\dirloc(G) \le k$.
\end{theorem}
\begin{proof}
Let $(T, \mathcal{B})$ be a restricted smooth tree decomposition of $G$ with width $k$. 
We give a strategy for $k$ cops to locate a robber on $G$.  Loosely, the cops' strategy works as follows.  Let $r$ be an arbitrary leaf in $T$.  The cops imagine $T$ as being rooted at $r$; they amass their forces in the vertices of $B_r$ and march through $T$, using probes to ``steer'' them in the direction of the robber, until they push the robber to a leaf of $T$ and, ultimately, locate him.  Throughout the proof we will sometimes tacitly assume that the cops never probe the vertex containing the robber, since if this happens they immediately win.

Let $r$ be a leaf of $T$ and let $r'$ denote its neighbor. The cops begin by probing all vertices of $B_r \cap B_{r'}$; note that there are exactly $k$ such vertices because $T$ is restricted a restricted smooth tree decomposition.  Note that since $\size{B_r} \le k+1$, all vertices in $B_r$ get probed except, perhaps, for one.  By Lemma \ref{lem:oneside}, all vertices consistent with the robber's response lie either in $B_r$ or in bags corresponding to other nodes of $T$.  In the former case, either the cops have probed the robber's vertex, or the robber occupies the lone unprobed vertex in $B_r$; either way, the cops win.  So, suppose otherwise.

More generally, for a node $a$ and child $b$ in $T$, we say that the cops have \textit{secured edge $ab$} if they have just probed $B_a \cap B_b$ 
and they have learned that all bags containing the robber's vertex correspond to nodes in the component of $T-ab$ containing $b$.  (At the current point in the game, the cops have just secured edge $rr'$.)  We will show that if the cops have just secured some edge $ab$, then within the next two rounds, they can either locate the robber or move farther down the tree -- that is, they will have secured edge $a'b'$ for some node $a'$ and child $b'$, where $\dist(r,a') > \dist(r,a)$.  Since $T$ is finite, the cops cannot keep moving farther from $r$ indefinitely, so this process must eventually terminate with the cops locating the robber.

Suppose the cops have just secured some edge $ab$, and consider the state of the game before the ensuing recontamination phase.  If $b$ is a leaf of $T$, then the robber must occupy some vertex in $B_b$.  The cops have just probed $B_a \cap B_b$, so they have thereby probed all but at most one vertex of $B_b$; hence, either they have probed the robber's vertex, or the robber occupies the lone unprobed vertex in $B_b$.  Either way, the cops win.

Suppose instead that $\deg(b) \ge 2$.  Consider the state of the game after the recontamination phase.  Let $c$ be a child of $b$ and, if $\deg(b) = 3$, let $d$ be the other child.  Let $T_a$ and $T_c$ denote the components of $T-b$ containing $a$ and containing $c$, respectively; if $b$ has two children, then let $T_d$ denote the component of $T-b$ containing $d$.  Let $V_a$ (resp. $V_c, V_d$) denote $\cup_{f \in V(T_a)} B_f \setminus B_b$ (resp. $\cup_{f \in V(T_c)} B_f \setminus B_b$, $\cup_{f \in V(T_d)} B_f \setminus B_b$).  Before recontamination, the cops knew that the robber occupied some vertex in $B_b \cup V_c \cup V_d$.  Moreover, he was not in $B_a \cap B_b$; since $B_a \cap B_b$ separates $V_a$ from the other vertices of $G$, the robber cannot have left $B_b \cup V_c \cup V_d$ with his last move.  

The cops now probe $B_b \cap B_c$.  Let $x$ denote the unprobed vertex of $B_b$.  By Lemma \ref{lem:oneside}, all vertices consistent with the robber's response lie either in $B_b \cup V_d$ or in $V_c$.  In the latter case, the cops have secured edge $bc$, as desired.  Suppose instead that the cops learn that the robber must be in $B_b \cup V_d$.  If $b$ has only one child, then $V_d = \emptyset$ and so the robber must occupy some vertex in $B_b$; because the cops have just probed all vertices of $B_b$ other than $x$, the robber must occupy $x$, and the cops have won.  Thus, suppose $b$ has two children.  If $B_b \cap B_c = B_b \cap B_d$, then the cops' last probe also told them whether the robber was in $B_b$ or in $V_d$; in the former case the robber must occupy vertex $x$, while in the latter case they have secured edge $bd$.  Suppose then that $B_b \cap B_c \not = B_b \cap B_d$, and note that because $(T, \mathcal{B})$ is a restricted smooth tree decomposition, it follows that $x \in B_b \cap B_d$.  Because the cops probed $B_b \cap B_c$, which separates $V_c$ from $V_a \cup V_d$, the robber cannot move into $V_c$ with his subsequent move; however, he could have occupied $x$, and thus could have moved from there to some neighbor of $x$.  Hence the robber occupies some vertex in $N[x] \cup B_b \cup V_d$.

The cops now probe $B_b \cap B_d$.  By Lemma \ref{lem:oneside}, this probe tells the cops whether or not the robber occupies some vertex in $V_d$.  If he does, then the cops have secured edge $bd$.  Otherwise, the robber must occupy some vertex in $N[x]$.  However, since $x \in B_b \cap B_d$, the cops have just probed $x$, so the robber's response to that probe uniquely determines his location, and the cops win.
\end{proof}

We remark that Theorem \ref{thm:treewidth_restricted} implies that $\dirloc(G)$ is bounded above by the \textit{pathwidth} of $G$, i.e. the minimum width of a tree decomposition in which the tree $T$ is a path.  This is so because given an optimal smooth path decomposition of $G$, one can simply merge identical (necessarily adjacent) bags to obtain a restricted smooth tree decomposition of $G$ whose width is the pathwidth of $G$.  

Additionally, every graph $G$ has a $(\tw(G)+1)$-restricted tree decomposition: we can obtain one by starting with a smooth tree decomposition of $G$ with width $\tw(G)$ and repeatedly splitting vertices of degree greater than 3.  This observation yields an alternative proof that $\dirloc(G) \le \tw(G)+1$ (as already shown in Corollary \ref{cor:treewidth_degen}).  

We are unaware of any graphs $G$ for which $\dirloc(G) = \tw(G)+1$, and in fact we suspect that $\dirloc(G) \le \tw(G)$ for all $G$.  Unfortunately, this cannot be obtained as an immediate consequence of Theorem \ref{thm:treewidth_restricted}, because not every graph $G$ has a restricted smooth tree decomposition of width $\tw(G)$.  However, for graphs with small treewidth, one can always find such a tree decomposition; in particular, if $\tw(G) \le 2$, then $\dirloc(G) \le \tw(G)$.  As a first step toward proving this, we state a useful lemma.  Recall that given a graph $G$ and a vertex $v$ in $G$, the operation of \textit{splitting} $v$ entails replacing $v$ with two new adjacent vertices $v_1$ and $v_2$ and, for each $u \in N_G(v)$, adding either edge $uv_1$ or edge $uv_2$.  (Refer to Figure \ref{fig:treewidth_split}.)

\begin{center}
\begin{figure}
\,\hfill
\begin{tikzpicture}
[anchor=base, baseline, inner sep=0mm, semithick,
 vertex/.style={draw=gray, fill=white, circle, inner sep=2pt, minimum size=0.8cm},
 vertexlabel/.style={draw=none, fill=none, shape=rectangle, inner sep=2pt, font=\small},
 textnode/.style={draw=none, fill=none, shape=rectangle, inner sep=2pt},
 xscale=3,yscale=3.35]

\node (v) at (0,0) [vertex] {$v$};
\node (w) at (-0.5,0) [vertex] {$w$};
\node (x) at (-0.25,0.4) [vertex] {$x$};
\node (y) at (0.25,0.4) [vertex] {$y$};
\node (z) at (0.5,0) [vertex] {$z$};
\node (invis) at (0,-0.2) {};

\draw (w) -- (v) -- (x);
\draw (y) -- (v) -- (z);
\end{tikzpicture}
\quad\quad\quad{\Large $\mathbf{\longrightarrow}$}\quad\quad\quad
\begin{tikzpicture}
[anchor=base, baseline, inner sep=0mm, semithick,
 vertex/.style={draw=gray, fill=white, circle, inner sep=2pt, minimum size=0.8cm},
 vertexlabel/.style={draw=none, fill=none, shape=rectangle, inner sep=2pt, font=\small},
 textnode/.style={draw=none, fill=none, shape=rectangle, inner sep=2pt},
 xscale=3,yscale=3.35]

\node (v1) at (-0.25,0) [vertex] {$v_1$};
\node (v2) at (0.25,0) [vertex] {$v_2$};
\node (w) at (-0.75,0) [vertex] {$w$};
\node (x) at (-0.25,0.4) [vertex] {$y$};
\node (y) at (0.25,0.4) [vertex] {$x$};
\node (z) at (0.75,0) [vertex] {$z$};

\draw (w) -- (v1) -- (x);
\draw (y) -- (v2);
\draw (v1) -- (v2) -- (z);
\end{tikzpicture}
\,\hfill\,
\caption{Splitting vertex $v$ into $v_1$ and $v_2$; the neighbors of $v$ may be partitioned between $v_1$ and $v_2$ arbitrarily.}
\label{fig:treewidth_split}
\end{figure}
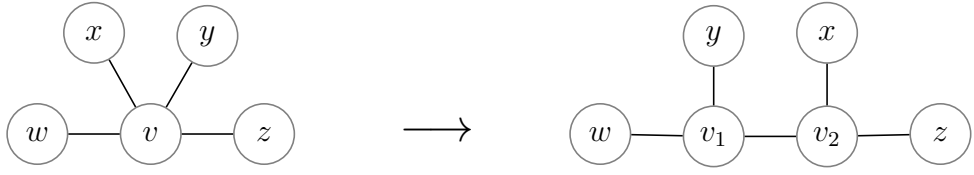
\end{center}

\begin{lemma}\label{lem:treewidth_split}
Let $(T, \mathcal{B})$ be a tree decomposition of a graph $G$.  If $T'$ is obtained from $T$ by splitting some $a \in V(T)$, and $\mathcal{B}'$ is formed from $\mathcal{B}$ by removing $B_a$ and adding a copy of $B_a$ for each of the two new vertices in $T'$, then $(T', \mathcal{B}')$ is a tree decomposition of $G$ with the same width as $(T, \mathcal{B})$. 
\end{lemma}
\begin{proof}
Fix $a \in V(T)$, let $N(a) = \{b_1, \dots, b_m\}$, and suppose that $T'$ is formed by splitting $a$ into $a_1$ and $a_2$, with $B_{a_1} = B_{a_2} = B_a$.  Note that every bag used in $T'$ was also used in $T$; it follows that the width of $T'$ is the same as that of $T$ and that for every edge $uv$ in $G$, there is a bag in $\mathcal{B}'$ containing both $u$ and $v$.  

Now consider any nodes $c$ and $d$ in $T'$ with $v \in B_c \cap B_d$ for some $v \in V(G)$.  The unique $c,d$-path in $T'$ corresponds to an analogous path in $T$, in which the vertex $a$ (if it appears at all) has been replaced with one or both of $a_1$ and $a_2$.  Since $(T, \mathcal{B})$ is a tree decomposition of $G$, all bags corresponding to vertices along the path in $T$ contain $v$, and hence so do all bags corresponding to vertices along the path in $T'$.  Consequently, $(T', \mathcal{B}')$ is a tree decomposition of $G$, as claimed. 
\end{proof}

\begin{theorem}\label{thm:small_treewidth}
If $\tw(G) \le 2$, then $\dirloc(G) \le \tw(G)$.
\end{theorem}
\begin{proof}
If $\tw(G) = 1$, then $G$ must be a tree, so Theorem \ref{thm:chordal} implies that $\dirloc(G) \le 1$.

Suppose instead that $\tw(G) = 2$.  It suffices to show that $G$ has a restricted smooth tree decomposition of width 2; the result will then follow by Theorem \ref{thm:treewidth_restricted}.  Let $(T, \mathcal{B})$ be a tree decomposition of $G$ having width 2.  Without loss of generality, we may further assume that $(T, \mathcal{B})$ is a smooth tree decomposition.  

We will show how to modify $(T, \mathcal{B})$ to produce a restricted smooth tree decomposition of width 2.  To do this, we must ensure that the tree underlying the new decomposition has no node of degree exceeding 3. 
%
We begin by modifying $(T, \mathcal{B})$ to obtain a decomposition in which every node whose degree exceeds 3 corresponds to a bag of size 2.  Suppose $T$ has a node $a$ with $\deg(a) > 3$ and $\size{B_a} = 3$.  Let $B_a = \{u, v, w\}$.  Because $(T, \mathcal{B})$ is a smooth tree decomposition, for any neighbor $b$ of $a$, the bag $B_b$ contains exactly two of $u$, $v$, and $w$.  Let 
\[\mathcal{F}_{uv} = \{b \in N(a) \, : \, w \not \in B_b\}, \,\, \mathcal{F}_{uw} = \{b \in N(a) \setminus \mathcal{F}_{uv} \, : \, v \not \in B_b\}, \,\, \text{ and } \,\, \mathcal{F}_{vw} = N(a) \setminus (\mathcal{F}_{vw} \cup \mathcal{F}_{uw}).\]
Note that $\mathcal{F}_{uv}$, $\mathcal{F}_{uw}$, and $\mathcal{F}_{vw}$ partition $N(a)$; moreover, for all $b \in \mathcal{F}_{uv}$ (resp. $\mathcal{F}_{uw}, \mathcal{F}_{vw}$) we have $w \not \in B_b$ (resp. $v \not \in B_b, u \not \in B_b\}$.

We now modify the decomposition as follows.  Delete all edges incident to node $a$.  Add new nodes $a_{uv}$, $a_{uw}$, and $a_{vw}$, each adjacent to $a$.  Let $B_{a_{uv}} = \{u,v\}$, $B_{a_{uw}} = \{u,w\}$, and $B_{a_{vw}} = \{v,w\}$. Additionally, add an edge from $a_{uv}$ (resp. $a_{uw}, a_{vw}$) to each node in $\mathcal{F}_{uv}$ (resp. $\mathcal{F}_{uw}$, $\mathcal{F}_{vw}$).  (Refer to Figure \ref{fig:treewidth_uvw}.)  

\begin{figure}
\begin{tikzpicture}
[anchor=base, baseline, inner sep=0mm, semithick,
 vertex/.style={draw=gray, fill=white, ellipse, inner sep=2pt, minimum size=0.35cm},
 vertexlabel/.style={draw=none, fill=none, shape=rectangle, inner sep=2pt, font=\small},
 textnode/.style={draw=none, fill=none, shape=rectangle, inner sep=2pt},
 xscale=2.2,yscale=2]

\node (a) at (0,0) [vertex, draw=black, ultra thick, fill=gray!20] {\footnotesize $\{u,v,w\}$};
\node (atext) at ($(a)+(0,-0.2)$) [vertexlabel, anchor=north] {\small $\mathbf{a}$};
\node (b) at (-1,0) [vertex] {\footnotesize $\{u,v,x\}$};
\node (c) at (-0.7,0.5) [vertex] {\footnotesize $\{u,x,y\}$};
\node (d) at (0,0.8) [vertex] {\footnotesize $\{u,v,z\}$};
\node (e) at (0.7, 0.5) [vertex] {\footnotesize $\{u,w,x\}$};
\node (f) at (1,0) [vertex] {\footnotesize $\{u,w\}$};
\node (bn1) at (-1.6,0.05) {};
\node (bn2) at (-1.5,0.3) {};
\node (cn) at (-0.9,0.9) {};
\node (dn) at (0,1.25) {};
\node (en1) at (0.8,0.93) {};
\node (en2) at (1.2,0.8) {};
\node (fn) at (1.6,0.05) {};

\draw (a) -- (b) -- (bn1);
\draw (b) -- (bn2);
\draw (a) -- (c) -- (cn);
\draw (a) -- (d) -- (dn);
\draw (a) -- (e) -- (en1);
\draw (e) -- (en2);
\draw (a) -- (f) -- (fn);
\end{tikzpicture}
\quad{\Large $\mathbf{\longrightarrow}$}\quad
\begin{tikzpicture}
[anchor=base, baseline, inner sep=0mm, semithick,
 vertex/.style={draw=gray, fill=white, ellipse, inner sep=2pt, minimum size=0.35cm},
 vertexlabel/.style={draw=none, fill=none, shape=rectangle, inner sep=2pt, font=\small},
 textnode/.style={draw=none, fill=none, shape=rectangle, inner sep=2pt},
 xscale=2.2,yscale=2]

\node (a) at (0,0.1) [vertex, draw=black, ultra thick, fill=gray!20] {\footnotesize $\{u,v,w\}$};
\node (atext) at ($(a)+(0,-0.2)$) [vertexlabel, anchor=north] {\small $\mathbf{a}$};
\node (auv) at (-0.6,-0.35) [vertex, draw=black, ultra thick, fill=gray!20] {\footnotesize $\{u,v\}$};
\node (auvtext) at ($(auv)+(0.1,-0.2)$) [vertexlabel, anchor=north] {\small $\mathbf{a_{uv}}$};
\node (avw) at (0.6,-0.35) [vertex, draw=black, ultra thick, fill=gray!20] {\footnotesize $\{v,w\}$};
\node (avwtext) at ($(avw)+(0,-0.2)$) [vertexlabel, anchor=north] {\small $\mathbf{a_{vw}}$};
\node (auw) at (0,0.65) [vertex, draw=black, ultra thick, fill=gray!20] {\footnotesize $\{u,w\}$};
\node (auwtext) at ($(auw)+(0.3,-0.1)$) [vertexlabel, anchor=north west] {\small $\mathbf{a_{uw}}$};
\node (b) at (-1.2,-0.9) [vertex] {\footnotesize $\{u,v,x\}$};
\node (c) at (-1.5,-0.35) [vertex] {\footnotesize $\{u,x,y\}$};
\node (d) at (-1.0,0.1) [vertex] {\footnotesize $\{u,v,z\}$};
\node (e) at (-0.5, 1.2) [vertex] {\footnotesize $\{u,w,x\}$};
\node (f) at (0.5, 1.2) [vertex] {\footnotesize $\{u,w\}$};
\node (bn1) at (-1.7,-1.1) {};
\node (bn2) at (-1.75,-0.7) {};
\node (cn) at (-1.75,0.05) {};
\node (dn) at (-1.05,0.5) {};
\node (en1) at (-0.7,1.6) {};
\node (en2) at (-0.3,1.6) {};
\node (fn) at (0.7,1.6) {};
\node (invis) at (0,-1.2) {};

\draw (a) -- (auv) -- (b) -- (bn1);
\draw (b) -- (bn2);
\draw (a) -- (auv) -- (c) -- (cn);
\draw (a) -- (auv) -- (d) -- (dn);
\draw (a) -- (auw) -- (e) -- (en1);
\draw (e) -- (en2);
\draw (a) -- (auw) -- (f) -- (fn);
\draw (a) -- (avw);
\end{tikzpicture}
\caption{Partitioning the neighbors of $a$ between $a_{uv}$, $a_{uw}$, and $a_{vw}$.}
\label{fig:treewidth_uvw}
\end{figure}
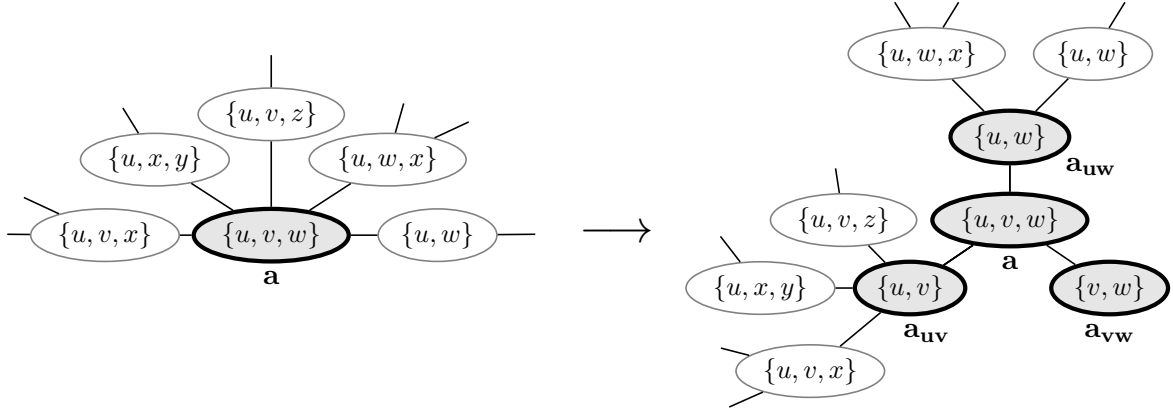

Refer to the resulting tree decomposition as $(T', \mathcal{B}')$.  To show that this is a valid tree decomposition, we must verify the two properties in Definition \ref{defn:treewidth}.  Property (ii) is clear, because every bag corresponding to a node in $T$ also corresponds to the same node in $T'$.  For property (i), fix $x \in V(G)$ and let $b$ and $c$ be nodes in $T'$ such that $x \in B_b \cap B_c$; we must show that $x \in B_d$ for every node $d$ on the unique $b,c$-path in $T'$.  If the $b,c$-path in $T'$ does not pass through $a$, then either the path does not use any vertex in $\{a_{uv}, a_{uw}, a_{vw}\}$, or the path ends at one of these vertices.  In the former case, the $b,c$-path in $T'$ is identical to the $b,c$-path in $T$, while in the latter case it is the same as the $b,a$-path in $T'$ except with $a$ replaced by one of $a_{uv}$, $a_{uw}$, and $a_{vw}$; in both cases, the desired property holds due to $T$ being a tree decomposition.  

Suppose instead that the $b,c$-path in $T'$ passes through $a$.  It suffices to argue that the claim holds when $c = a$, since any path passing through $a$ can be split into two paths having $a$ as an endpoint.  Suppose that the $b,a$-path in $T'$ passes through $a_{uv}$; the other two cases are similar.  The $b,a$-path in $T$ is identical to the path in $T'$, except that the latter includes $a_{uv}$ immediately prior to $a$; hence it suffices to argue that $x \in B_{a_{uv}}$.  Let $d$ denote the neighbor of $a$ in the $b,a$-path in $T$.  Because the $b,a$-path in $T'$ passes through $a_{uv}$, we know that $d \in \mathcal{F}_{uv}$; hence $w \not \in B_{d}$ and so $x \not = w$.  We also know that $x \in B_a = \{u,v,w\}$.  Thus $x \in \{u,v\} = B_{a_{uv}}$, which completes the proof that $T'$ is a tree decomposition.

Note that all of the new nodes introduced in $T'$ correspond to bags of size 2, and that the original node $a$ has degree 3 in $T'$; hence, compared to $T$, the decomposition $T'$ has fewer vertices with degree exceeding 3 that correspond to bags of size 3.  Repeating this process will eventually produce a tree decomposition in which every node with degree exceeding 3 corresponds to a bag of size at most 2.  Additionally, note that this modification has maintained the property for any pair of adjacent nodes in $T'$, the corresponding bags intersect in exactly two vertices of $G$.  

Finally, we eliminate nodes of degree exceeding 3.  Let $a$ be a node with degree greater than 3, and let $N(a) = \{b_1, \dots, b_k\}$.  Split $a$ into two new vertices $a_1$ and $a_2$, with $a_1$ adjacent to $b_1$ and $b_2$, and $a_2$ adjacent to $b_3, \dots, b_k$; by Lemma \ref{lem:treewidth_split}, this yields a new tree decomposition of $G$ with the same width.  Moreover, in the new decomposition, $a_1$ has degree 3 and $a_2$ has degree $k-1$, so this operation has decreased the sum $\sum_{v} \max\{0, \deg(v)-3\}$.  Repeating this process, we can eventually produce a tree decomposition in which $\sum_{v} \max\{0, \deg(v)-3\} = 0$; such a decomposition necessarily has maximum degree at most 3.  Furthermore, every vertex that gets split corresponds to a bag with at most two elements, so once again we have preserved the property that adjacent nodes correspond to bags that share exactly two vertices of $G$.  Finally, the new decomposition clearly has width 2, since none of the operations we have applied have changed the width of the decomposition.  Hence we have produced a restricted smooth tree decomposition of $G$ with width 2, as desired. 
\end{proof}

\begin{corollary}
If $G$ is outerplanar, then $\dirloc(G) \le 2$.
\end{corollary}
\begin{proof}
This is immediate from Theorem \ref{thm:small_treewidth} and the fact that outerplanar graphs have treewidth at most 2.
\end{proof}

\section{Projective planes}\label{sec:dirloc_proj_planes}


In this section, we consider the directional localization game on incidence graphs of finite projective planes.  Recall that a \textit{finite projective plane} is a collection of \textit{points} and \textit{lines}, with each line being a set of points, such that:
\begin{itemize}
\item [\textbf{(1)}] any two distinct lines contain exactly one common point; 
\item [\textbf{(2)}] any two distinct points belong to exactly one common line; and
\item [\textbf{(3)}] there exist four distinct points, of which no line contains more than two.
\end{itemize}
It is well-known that these axioms imply, as a consequence, that every line contains exactly $q+1$ points and every point lies on exactly $q+1$ lines, for some $q$; we refer to $q$ as the \textit{order} of the projective plane.  

Given a finite projective plane, define its \textit{incidence graph} to be the graph having one vertex for each point, one vertex for each line, and an edge joining a point $p$ with a line $\ell$ if and only if $p \in \ell$.  We will refer to the vertices corresponding to points as \textit{point vertices} and those corresponding to lines as \textit{line vertices}, with $\mathcal{P}$ denoting the set of all point vertices and $\mathcal{L}$ the set of line vertices.  We will sometimes refer to \textit{duality}; this is the principle that because the axioms of a projective plane are symmetric with respect to points and lines, taking any statement about a projective plane and swapping the terms ``point'' and ``line'' yields a logically equivalent statement.  (The careful reader may notice that axiom (3) is not inherently symmetric with respect to points and lines; however, it can be shown that in any projective plane there exist four lines of which no point lies on more than two.)

If $G$ is the incidence graph of a projective plane, then it is easily seen that $G$ is a bipartite graph with diameter 3 and that any two vertices in the same partite set of $G$ have exactly one common neighbor (and, in particular, lie at distance 2); we will use these facts repeatedly in what follows.

Numerous graph searching and pursuit-evasion parameters have been studied on incidence graphs of projective planes, and in many cases, these graphs serve as enlightening examples of ``worst-case'' behavior.  For example, in the game of Cops and Robbers, incidence graphs of projective planes are one of the best-known families of Meyniel-extremal graphs, i.e. graphs for which $c(G) = \Theta(\sqrt{n})$: for a projective plane of order $q$, the incidence graph has $2(q^2+q+1)$ vertices and cop number $q+1$ \cite{Pra10}.  In the directional localization game, however, the situation is quite different: as we will show, on the incidence graph of a projective plane, both $\dirlocinc$ and $\dirloccomp$ are 2.  In proving this fact, we will present a cop strategy that is somewhat complex; to simplify the presentation, we begin with several auxiliary lemmas.

\begin{lemma}\label{lem:proj_plane_adj_vxs}
Let $G$ be the incidence graph of a projective plane 
and consider the partial-feedback game on $G$ played with two cops.  If at any point during the game the cops probe adjacent vertices $x$ and $y$, and the probe at $x$ points to $y$ (or vice-versa), then the cops can locate the robber.
\end{lemma}
\begin{proof}
It suffices to consider the case where, prior to the cops' probe, all vertices of $G$ are contaminated; if the cops can win in this case, then they can clearly win no matter which vertices of $G$ are actually contaminated.

Let $x'$ (respectively, $y'$) denote the robber's response to the probe at $x$ (resp. $y$), and assume that either $x' = y$ or $y' = x$.  By symmetry, let us assume $x \in \mathcal{P}$, $y \in \mathcal{L}$, and $y' = x$.  We may also suppose that the robber occupies neither $x$ nor $y$, since then the cops would win immediately.  $G$ has diameter 3, so the distance from $y$ to the robber is at most 3; hence, the distance from $x$ to the robber is at most 2, and the distance from $x'$ to the robber is at most 1.  Consequently, the only line vertex consistent with the probes -- and thus the only contaminated line vertex -- is $x'$, and the set of consistent point vertices is $N(x') - x$.  

After the ensuing recontamination phase, the set of contaminated point vertices is $N(x')$.  On the next cop turn, one cop probes $x'$, while the other cop probes any uncontaminated vertex $w$ in $N(x)$.  (Note that some such $w$ must exist:  the neighbors of $x'$ all have $x'$ as a common neighbor, so they cannot have any other common neighbors; since $x$ was clear prior to recontamination, its neighbors other than $x'$ cannot have been recontaminated.)  As usual, suppose that the robber does not occupy either vertex probed by the cops.  If the robber occupies a point vertex, then the probe at $x'$ must point to that vertex.  If instead he occupies a line vertex, then both $x'$ and $w$ are at distance 2 from the robber's vertex.  Hence the probe at $x'$ and the probe at $w$ must both point to neighbors of the robber's vertex, so the probe responses' unique common neighbor is the only contaminated line vertex consistent with the probes.  (Note that the responses to the probes are necessarily distinct -- and thus have only one common neighbor -- because the probed vertices share $x$ as a common neighbor and no neighbors of $x$ are contaminated at the time of probing except for $x'$.)

Hence, the cops have narrowed down the robber's possible to location to two adjacent vertices, say $u \in \mathcal{P}$ and $v \in \mathcal{L}$.  In the ensuing recontamination phase, all neighbors of $u$ and $v$ become contaminated.  Finally, on the cops' following turn, they probe $u$ and $v$; suppose that the robber does not occupy either of these vertices.  If the robber occupies a point vertex, then he must occupy some neighbor of $v$ and, moreover, he must be at distance 2 from $u$; thus the probe at $u$ must point to $v$ (since there is only one path of length 2 from $u$ to any given vertex at distance 2), and the probe at $v$ must point to the robber's location.  Likewise, if the robber occupies a line vertex, then the probe at $v$ must point to $u$ and the probe at $u$ must point to the robber's location.  In either case there is only one possible location for the robber, and the cops can clearly determine which case has occurred, so they have successfully located the robber.
\end{proof}

\begin{lemma}\label{lem:2_neighborhoods}
	Let $G$ be the incidence graph of a finite projective plane. If the cops probe  $x \in\mathcal{P}$ and $y \in \mathcal{L}$ with responses $x'$ and $y'$ respectively, then all vertices consistent with these probes lie in  $N(x') \cup N(y')$.
\end{lemma}

\begin{proof}
	Let the cops probe a point vertex $x$ and a line vertex $y$ and receive responses of $x'$ and $y'$ respectively.  
    Suppose the robber is on some vertex $r$.  We assume $r\notin \{x,y\}$, since otherwise the cops have located the robber. If $r$ is a point vertex, then it shares a unique neighbor with $x$; this neighbor must necessarily be $x'$, and therefore $r \in N(x')$.  Thus every point vertex that could contain the robber lies in $N(x')$. By duality, a similar argument yields that every line vertex that may contain the robber lies in $N(y')$.  Thus, all vertices consistent with the probes lie in $N(x') \cup N(y')$, as claimed.
\end{proof}

Broadly, the cops' strategy for the partial-feedback game on the projective plane will be to probe so as to gradually shrink the number of contaminated vertices in a given partite set.  Our next three lemmas provide a framework by which they can accomplish this.

\begin{lemma}\label{lem:reduction_w_adjacent_sources}
	Let $G$ be the incidence graph of a finite projective plane of order $q$, and consider the partial-feedback game on $G$.  Suppose that at the end of a probing phase, there exist vertices $x \in \mathcal{P}$ and $y \in \mathcal{L}$ such that:
	\begin{enumerate}
		\item[(1)] all contaminated vertices are contained within $N(x) \cup N(y)$;
		\item[(2)] $x$ and $y$ are adjacent; and
		\item[(3)] $y$ is contaminated.
	\end{enumerate}
	Additionally, let $i$ denote the number of contaminated line vertices.  If $i \ge 2$, then either the cops can win the game, or they can ensure at the end of the following probing phase, all contaminated vertices will be contained in $N(u) \cup N(v)$ for some non-adjacent point vertex $u$ and line vertex $v$, and there will be at most $i-1$ contaminated point vertices.
\end{lemma}

\begin{proof}

    Let $N(x) = \{w_1, \dots, w_{q+1}\}$, where $w_1, \dots, w_{i}$ are contaminated and $w_{i+1}, \dots, w_{q+1}$ are not.  For each $j$, let $W_j = N(w_j) \setminus \{x\}$.  We claim that $\{\{x\}, W_1, \dots, W_{q+1}\}$ is a partition of $\mathcal{P}$.  To see this, consider an arbitrary point vertex $u$.  Either $u = x$, or 
    $u$ and $x$ have a unique common neighbor $z$.  Since $z$ is a neighbor of $x$, we have $z = w_j$ for some $j$, and so $u \in W_j$.  Additionally, $x$ is the unique common neighbor of any two of the $w_j$, the sets $W_j$ must be disjoint.
    
    Since $y$ is a contaminated line vertex adjacent to $x$, without loss of generality we may suppose that $y = w_1$.  Since all contaminated point vertices belong to$N(y)$, it follows that all vertices in $W_j$ for $j \geq 2$ are clear.
	After the ensuing recontamination phase, the set $\{x\} \cup W_1 \cup \dots \cup W_i$ will only contain contaminated vertices and the set $W_{i+1} \cup \dots \cup W_{q+1}$ will only contain clear vertices.

	Now consider the following cop turn.  The cops probe a point vertex $a$ in $W_2$ and the line vertex $b = w_2$; let $a'$ and $b'$, respectively, be the robber's responses to these probes.  If $a' = b$, then the cops can win the game via Lemma \ref{lem:proj_plane_adj_vxs}.  Suppose instead that $a' \neq b$.  By Lemma \ref{lem:2_neighborhoods}, every point vertex consistent with the probes must lie in $N(a')$.  
    Each $w_j$ has one common neighbor with $a'$, so there is a unique vertex $a_j \in W_j \cap N(a')$.  For $j > i$, we know that $a_j$ was clear before the cops' move; additionally, since $a_2 = a$, vertex $a_2$ was cleared by virtue of being probed.  Hence there are at most $i-1$ contaminated point vertices, all of which lie in $N(a')$. Additionally, by Lemma \ref{lem:2_neighborhoods}, the set of contaminated vertices is contained in $N(a') \bigcup N(b')$.  Finally, since $a$ and $b$ are adjacent, it follows that $a'$ and $b'$ must not be adjacent, since otherwise $a$ and $b'$ would have two common neighbors (namely $a'$ and $b$).  Thus the game has reached a state of the desired form; this completes the proof. 
\end{proof}

\begin{lemma}\label{lem:reduction_w_nonadjacent_case_1}
	Let $G$ be the incidence graph of a finite projective plane of order $q$, and consider the partial-feedback game on $G$.  Suppose that at the end of a probing phase, there exist vertices $x \in \mathcal{P}$ and $y \in \mathcal{L}$ such that:
	\begin{enumerate}
        \item[(1)] all contaminated vertices are contained within $N(x) \cup N(y)$; 
		\item[(2)] $x$ and $y$ are nonadjacent; and
		\item[(3)] there exist a clear vertex $\ell \in N(x)$ and a contaminated vertex $m \in N(y)$ such that $\ell m \in E(G)$.
	\end{enumerate}
    Additionally, let $i$ denote the number of contaminated line vertices. If $i \ge 2$, then the cops can guarantee that after some subsequent probing phase, all contaminated vertices will be contained in $N(u) \cup N(v)$ for some non-adjacent point vertex $u$ and line vertex $v$, and there will be at most $i-1$ contaminated vertices in some partite set of $G$.
\end{lemma}

\begin{proof}
    Note that by conditions (1) and (2), the vertices $x$ and $y$ must not, themselves, be contaminated. Let $N(x) = \{w_1, \dots, w_{q+1}\}$, where $w_1, \dots, w_{i}$ are contaminated and $w_{i+1}, \dots, w_{q+1}$ are clear.  Each $w_j$ has exactly one common neighbor with $y$, and the $w_j$ all share $x$ as a common neighbor, so no two $w_j$ are adjacent to the same neighbor of $y$.  Thus, we may let $N(y) = \{u_1, \dots, u_{q+1}\}$ such that $w_ju_j \in E(G)$ for all $j$.  By condition (3), we may suppose without loss of generality that $w_{q+1}$ is clear and $u_{q+1}$ is contaminated.
    For each $j$, let $W_j = N(w_j) \setminus \{x\}$. 
    After the recontamination phase, all contaminated point vertices will lie in $N(y) \cup \{x\} \cup W_1 \cup \dots \cup W_i$.  
    
    Let $a$ be any vertex in $W_{q+1}$ other than $u_{q+1}$, and let $b = w_1$.  On the the cops' next turn, they probe $a$ and $b$; let the robber's responses to these probes be $a'$ and $b'$, respectively.  We consider three cases. 
    
    \medskip
    \textbf{Case 1:} $a' = w_{q+1}$. By Lemma \ref{lem:2_neighborhoods}, all point vertices consistent with the probes must lie in $N(a')$. This leaves us with only two potentially contaminated point vertices, namely $u_{q+1}$ and $x$.  We consider two subcases:
    
    \textbf{Case 1a:} $b' \neq x$.  Since the probe at $b$ pointed to a vertex other than $x$, we see that $x$ is not consistent with this probe.  Therefore, the only remaining contaminated point vertex is $u_{q+1}$.  Vertices $a'$ and $b$ share a single common neighbor, namely $x$. Since $b'\neq x$, we see that $a'$ and $b'$ are non-adjacent (since otherwise both $x$ and $b'$ would be common neighbors of $a'$ and $b$).  Thus, by Lemma \ref{lem:2_neighborhoods}, all contaminated vertices are contained in $N(a') \cup N(b')$ for the non-adjacent point vertex $b'$ and line vertex $a'$, and there is only one contaminated point vertex (hence in particular there are at most $i-1$ contaminated point vertices, as claimed).

    \textbf{Case 1b:} $b' = x$.  In this case, $x$ is consistent with the probe at $b$.  Since $N(a') = \{x\} \cup  W_{q+1}$, we see that there are only two contaminated point vertices, namely $u_{q+1}$ and $x$.  However, note that all contaminated vertices lie in $N(a') \cup N(b')$ for adjacent vertices $a'$ and $b'$ with $b'$ being contaminated.  Applying the dual of Lemma \ref{lem:reduction_w_adjacent_sources} -- that is, swapping the roles of points and lines in Lemma \ref{lem:reduction_w_adjacent_sources} -- shows that after an additional probing phase, the cops can either win the game or ensure that all contaminated vertices lie in $N(u) \cup N(v)$ for some nonadjacent point vertex $u$ and line vertex $v$, and there is a single contaminated line vertex (and in particular, there are at most $i-1$ contaminated line vertices, as claimed).

    \medskip
    
    \textbf{Case 2:} $a' \neq w_{q+1}$ and $b'$ is the unique neighbor of $a'$ in $W_{1}$.  By Lemma \ref{lem:2_neighborhoods}, all vertices consistent with the probes lie in $N(a')\cup N(b')$. Since $a'$ has exactly one neighbor in each $W_j$ and only those neighbors in $W_1, \dots, W_i$ were contaminated prior to the probe, we see that $a'$ has at most $i$ contaminated neighbors.  Thus we can apply Lemma \ref{lem:reduction_w_adjacent_sources} to find that after an additional probing phase, all contaminated vertices lie in $N(u)\cup N(v)$ for some nonadjacent point vertex $u$ and line vertex $v$, and there are at most $i-1$ contaminated line vertices.

    \medskip

    \textbf{Case 3:} $a' \neq w_{q+1}$ and $b'$ is not the unique neighbor of $a'$ in $W_1$.  Once again, since $a'$ has exactly one neighbor in each $W_j$ and only those neighbors in $W_1, \dots, W_i$ were contaminated, we again see that $a'$ has at most $i$ contaminated neighbors.  In this case, however, the unique neighbor of $a'$ in $W_1$ -- which must also be adjacent to $b$ -- is not consistent with the probe at $b$, and therefore at most $i-1$ contaminated point vertices remain.  Also, since $a'$ has a unique neighbor in $W_1$, and that neighbor is not $b'$, we see that $a'$ and $b'$ are non-adjacent, as needed.  This completes the proof.

\end{proof}

\begin{lemma}\label{lem:reduction_w_nonadjacent_case_2}
	Let $G$ be the incidence graph of a finite projective plane of order $q$, and consider the partial-feedback game on $G$.  Suppose that at the end of a probing phase, there exist vertices $x \in \mathcal{P}$ and $y \in \mathcal{L}$ such that:
	\begin{enumerate}
		\item[(1)] all contaminated vertices are contained within $N(x) \cup N(y)$; 
		\item[(2)] $x$ and $y$ are nonadjacent; 
		\item[(3)] there do not exist a clear vertex $\ell \in N(x)$ and a contaminated vertex $m \in N(y)$ such that $\ell m \in E(G)$; and
        \item[(4)] there is at least one contaminated point vertex.
	\end{enumerate}
    Additionally, let $i$ denote the number of contaminated line vertices.  Then either the cops can win the game, or they can ensure that the end of the following probing phase, all contaminated vertices will be contained within $N(u) \cup N(v)$ for some non-adjacent point vertex $u$ and line vertex $v$, and there will be at most $i-1$ contaminated point vertices.
\end{lemma}

\begin{proof}
    Let $N(x) = \{w_1, \dots, w_{q+1}\}$, where $w_1, \dots, w_i$ are contaminated and $w_{i+1}, \dots, w_{q+1}$ are clear.  As in the proof of the preceding lemma, we may let $N(y) = \{u_1, \dots, u_{q+1}\}$ such that $w_ju_j \in E(G)$ for all $j$.  For each $j$, let $W_j = N(w_j) \setminus \{x\}$ and $U_j = N(u_j) \setminus \{y\}$.  
    By assumption (4), there must be some $j$ such that $u_j$ is contaminated; by assumption (3), it follows that $w_j$ must also be contaminated.  Without loss of generality, suppose that $u_1$ and $w_1$ are both contaminated. After the subsequent recontamination phase, the only contaminated point vertices are those neighbors of $y$ that were originally contaminated and those vertices in $\bigcup\limits^{i}_{j=1} W_j$.

    Let $a$ be any vertex in $W_1$ other than $u_1$ and let $b = w_1$.  On the cops' next turn, they probe $a$ and $b$; let the robber's responses to these probes be $a'$ and $b'$, respectively.  If $a' = b$, then the cops can win the game via Lemma \ref{lem:proj_plane_adj_vxs}, so suppose otherwise.  As in the proof of Lemma \ref{lem:reduction_w_adjacent_sources}, $\{\{y\}, U_1, \dots, U_{q+1}\}$ is a partition of $\mathcal{L}$, so $a' \in U_k$ for some $k \neq 1$.  Since the only common neighbor of $a$ and $x$ is $b$, and since $a' \not = b$, it follows that $a' \not \in N(x)$.  Additionally, since all vertices in each $W_j$ have $w_j$ as their unique common neighbor, it follows that $a'$ has exactly one neighbor in each $W_j$.
    
    We claim that every contaminated neighbor of $a'$ must belong to $W_1 \cup \dots \cup W_i$.  Indeed, the only contaminated point vertices outside of $W_1 \cup \dots \cup W_i$ are those $u_j$ that were already contaminated prior to the recontamination phase.  Since $u_j \in W_j$, we need only consider the case $j \ge i+1$.  In this case, if $u_j$ was contaminated, then prior to recontamination we would have had an adjacent contaminated vertex $u_j \in N(y)$ and clear vertex $w_j \in N(x)$, contradicting assumption (4).  
    
    It follows that $a'$ has at most $i$ contaminated neighbors -- one in $W_j$ for each $j \in \{1, \dots, i\}$.  However, the unique neighbor of $a'$ in $W_1$ is $a$ itself, which was probed; we may assume that vertex to be clear.  Consequently, $a'$ has at most $i-1$ contaminated neighbors.  Finally, since $a$ and $b$ were chosen to be adjacent, it follows that $a'$ and $b'$ must not be adjacent (since then $a$ and $b'$ would have two common neighbors, namely $a'$ and $b$).  Thus, at the end of the cops' turn, all contaminated vertices are contained within $N(a') \cup N(b')$, there are at most $i-1$ contaminated point vertices, and $a'b' \notin E(G)$, as desired. 
\end{proof}

\begin{lemma}\label{lem:reduction_overall}
    Let $G$ be the incidence graph of a finite projective plane of order $q$, and consider the partial-feedback game on $G$.  Suppose that at the end of a probing phase, all contaminated vertices are contained within $N(x) \cup N(y)$ for some non-adjacent $x \in \mathcal{P}$ and $y \in \mathcal{L}$, and let $i$ denote the number of contaminated line vertices.  If $i > 2$, then the cops can either win or ensure that at the end of some subsequent probing phase, the contamination will be contained in $N(u) \cup N(v)$ for some non-adjacent point vertex $u$ and line vertex $v$, and one partite set will contain at most $i-1$ contaminated vertices.
\end{lemma}

\begin{proof}
    If there exist a clear vertex $\ell \in N(x)$ and a contaminated vertex $m \in N(y)$ such that $\ell m \in E(G)$, then the claim holds by Lemma \ref{lem:reduction_w_nonadjacent_case_1}; otherwise, it holds by Lemma \ref{lem:reduction_w_nonadjacent_case_2}.

\end{proof}

As noted earlier, the cops' high-level strategy is to gradually reduce the number of contaminated vertices in a given partite set.  Our next two lemmas show that if they can reduce to having only one contaminated vertex in some partite set, then they can win the game.

\begin{lemma}\label{lem:1pt_1ln}
    Let $G$ be the incidence graph of a finite projective plane of order $q$.  Suppose that at the end of a probing phase, there is a single contaminated point vertex $x$ and a single contaminated line vertex $y$. Then the cops have a strategy to locate the robber.
\end{lemma}

\begin{proof}
    After the recontamination phase, all contaminated vertices will be contained within $N(x) \cup N(y)$.  On the cops' next turn, they probe $x$ and $y$; suppose they receive responses $x'$ and $y'$, respectively.
    
    \medskip
    
    \textbf{Case 1:} $xy \in E(G)$. For the cops to not immediately win by probing the robber's location, the robber must be located in $N(x) \setminus \{y\}$ or $N(y) \setminus \{x\}$.  Either way, one probe will be the response of the other probe and by Lemma \ref{lem:proj_plane_adj_vxs} the cops have a winning strategy.

    \medskip
    
    \textbf{Case 2:} $xy \notin E(G)$. By Lemma \ref{lem:2_neighborhoods}, the point vertices consistent with the probes are contained in $N(x')$, and the line vertices consistent with the probes are contained in $N(y')$. Thus, the only contaminated point vertex consistent with the probes is the unique common neighbor of $y$ and $x'$, say $a$.  Similarly, the only contaminated line vertex consistent with the probes is the unique common neighbor of $x$ and $y'$, say $b$.  If $x' \neq b$ or $y' \neq a$, then only a single contaminated vertex is consistent with the probe responses, so the cops win.  Assume instead that $x' = b$ and $y' = a$.  Now $a$ and $b$ are the only contaminated vertices, and since $a \in N(x') = N(b)$, it follows that $a$ and $b$ are adjacent; thus by Case 1, the cops have a winning strategy from this point in the game.
\end{proof}

\begin{lemma}\label{lem:1_contaminated_vtx}
    Let $G$ be the incidence graph of a finite projective plane of order $q$, and consider the partial-feedback game on $G$.  Suppose that at the end of a probing phase, there exist non-adjacent vertices $x \in \mathcal{P}$ and $y \in \mathcal{L}$ such that all contaminated vertices are contained within $N(x) \cup N(y)$.  If there is exactly one contaminated line vertex, then the cops have a winning strategy.
\end{lemma}

\begin{proof}
    Let $v$ denote the single contaminated line vertex.  For the sake of this argument, we will assume that all vertices in $N(y)$ are contaminated; it is clear that if the cops can win in this case, then they can win no matter which vertices of $N(y)$ are contaminated.
    
    Every line vertex has a common neighbor with $y$, so after the ensuing recontamination phase, all line vertices will be contaminated; the set of contaminated point vertices will be $N(y) \cup N(v)$. On the following cop turn, the cops will probe vertices $y$ and $v$; let $y'$ and $v'$, respectively, denote the robber's responses to these probes.  Note that if the robber occupies a point vertex, then he occupies some vertex in $N(y) \cup N(v)$, so either $y'$ or $v'$ (or both) must correspond to his location.  We consider two cases:\\
    
    
    \textbf{Case 1:} $y' = v'$.  In this case, the only contaminated point vertex consistent with the probe responses is $y'$; let $u = y'$.  The contaminated line vertices consistent with the probe responses are those vertices in $N(u)$, excluding $y$ and $v$ which were probed and therefore cleared.  After the ensuing recontamination phase, all point vertices will be contaminated except for those in $(N(y)\cup N(v)) - u$.  Also after recontamination, the set of contaminated line vertices will be $N(u)$.
    On the following cop turn, the cops shall probe $u$ and any other vertex $a \in N(y)$; let the probe responses be $u'$ and $a'$, respectively.  Note that if the robber occupies a line vertex, then that vertex belongs to $N(u)$; hence, the probe at $u$ must point to the robber's location.  Thus, $u'$ is the only contaminated line vertex that could possibly be consistent with the probes. 
    
    \textbf{Case 1a:} $u' = a'$.
    In this subcase, the unique common neighbor of $u$ and $a$ is $y$, so $u' = a' = y$.  As noted above, if the robber occupies a line vertex, then that vertex must be $y$; if instead he occupies a point vertex, then that vertex must be in $N(y)$.  Since all vertices of $N(y) - u$ were clear prior to the probing phase and since $u$ was probed (and thus cleared), there are no contaminated point vertices consistent with the probes.  Thus the only contaminated vertex consistent with the probes is $u'$, so the cops have located the robber.
    
    \textbf{Case 1b:} $u' \neq a'$.
    In this subcase, $u'$ and $a'$ have a unique common neighbor, say $b$.  This vertex $b$ is the only contaminated point vertex consistent with the probes.  If $u'=y$, then the cops win immediately: since $a$ is adjacent to $u'$, and $a' \not = u'$, the robber cannot occupy $u'$; as noted above, $u'$ was the only contaminated line vertex potentially consistent with the probes; since $u'$ in fact it is not consistent with the probes, the only possible location for the robber is $b$.  Otherwise, $b$ and $u'$ are the only contaminated vertices consistent with the probes, so the cops have a winning strategy by Lemma \ref{lem:1pt_1ln}. 

    \medskip

    \textbf{Case 2:} $y' \neq v'$.
    In this case, the only contaminated point vertices consistent with the probe responses will be $y'$ and $v'$.  Since $y'$ and $v'$ are distinct, they have a unique common neighbor $b$, which is the only contaminated line vertex consistent with the probes.  If either $y'$ or $v'$ is the common neighbor of $y$ and $v$, then that vertex would not be consistent with the probes because only one of the probes points there.  Hence only one contaminated point vertex and one contaminated line vertex would remain, so by Lemma \ref{lem:1pt_1ln} the cops would have a winning strategy.  Thus, we may assume that neither $y'$ nor $v'$ is the common neighbor of $y$ and $v$.
    
    After recontamination, the set of contaminated point vertices will be $N(b)$ and the set of contaminated line vertices will be $N(y') \cup N(v')$. On the cops' next turn, they will probe $b$ and any other line vertex $c$ such that the common neighbor of $b$ and $c$ is neither $y'$ nor $v'$.  Suppose they receive responses $b'$ and $c'$, respectively.  Since the set of contaminated point vertices prior to this move was $N(b)$, we see that $b'$ is the only contaminated point vertex consistent with the response of the probe at $b$.  Moreover, if $b' \not \in \{v',y'\}$, then there can be no contaminated line vertices consistent with the probes, so the robber would have to occupy $b'$; assume instead that $b' \in \{v', y'\}$.  With this assumption, we have $b' \not \in N(c)$, so $b' \neq c'$, and therefore $b'$ and $c'$ have a unique common neighbor.  This neighbor is the only contaminated line vertex consistent with the probes.  Thus only one contaminated point vertex and one contaminated line vertex remain, so the cops have a winning strategy by Lemma \ref{lem:1pt_1ln}.
\end{proof}

We are finally ready to prove the main result of this section.

\begin{theorem}\label{thm:proj_planes}
If $G$ is the incidence graph of a projective plane of order $q$, then $\dirloccomp(G) = \dirloc(G) = 2$. 
\end{theorem}
\begin{proof}
Let $G$ be the incidence graph of e a projective plane of order $q$.  Since $\dirloccomp(G) \le \dirloc(G)$, to show that $\dirloccomp(G) = \dirloc(G) = 2$, it suffices to argue that $\dirloccomp(G) > 1$ and $\dirloc(G) \le 2$.

To show that $\dirloccomp(G) > 1$, we give a strategy for the robber to avoid being located by a single cop in the full-feedback game on $G$.  Toward this end, we say that the game is in a \textit{stable state} if $G$ contains some adjacent vertices $x$ and $y$ such that all vertices in $N(x) \cup N(y)$ are contaminated; we claim that the robber can ensure that the game is in a stable state at the beginning of each round.  This is clearly the case at the beginning of the first round.  Suppose that the game is in a stable state at the beginning of the $k$th round; we show that the robber can ensure that it is in a stable state at the beginning of the $(k+1)$st round.  Suppose the cop probes vertex $p$.  By symmetry, we may suppose that $p$ and $x$ are point vertices, while $y$ is a line vertex.  We have two cases to consider. 

\medskip

\textbf{Case 1:} $p \not \in N(y)$.  Since $p$ and $x$ are both point vertices, they have a unique common neighbor $v$; the robber responds to the probe with $v$.  Since $v$ is the only common neighbor of $p$ and $x$, the only shortest path between them uses $v$.  Hence, the robber's response to the probe is consistent with his occupying either $v$ or $x$.  At the beginning of the $(k+1)$st round -- that is, after recontamination -- all vertices in $N(v) \cup N(x)$ will be contaminated.  Since $v$ and $x$ are adjacent, the game is in a stable state, as desired.

\medskip

\textbf{Case 2:} $p \in N(y)$.  In this case, the robber responds to the cop's probe with $y$.  Let $v$ be any neighbor of $y$ other than $p$.  Since $p$ and $v$ are both point vertices, they have a unique common neighbor, namely $y$.  Thus, the robber's response to the probe is consistent with his occupying either $y$ or $v$; at the beginning of the $(k+1)$st round, all vertices in $N(y) \cup N(v)$ will be contaminated, and the game is once again in a stable state.

\medskip

In either case, the robber can ensure that the game is in a stable state at the beginning of each round.  Note that if the cop were ever able to uniquely determine the robber's position -- say, she determined that the robber was located at vertex $z$ -- then at the beginning of the next round, the set of contaminated vertices would be precisely $N[z]$, so the game would not be in a stable state.  Hence, the robber's strategy ensures that he can never be located.  It follows that $\dirloccomp(G) > 1$.

Next, we argue that $\dirloc(G) \le 2$ by giving a strategy for two cops to locate a robber on $G$. On the cops' first move of the game, they probe a point vertex $x$ and an adjacent line vertex $y$. If the probe at $x$ points at $y$ (or vice-versa), then the cops can win by Lemma \ref{lem:proj_plane_adj_vxs}, so suppose otherwise.  Since $x$ and $y$ are adjacent, it is clear that the probe responses are non-adjacent. By lemma \ref{lem:2_neighborhoods}, all vertices consistent with these probes lie in $N(x') \cup N(y')$, where $x'$ is the response to the probe at $x$ and $y'$ is the response to the probe at $y$.

Suppose that there are $i$ contaminated vertices in the partite set containing the fewest contaminated vertices.  If $i=1$ then the cops can win by Lemma \ref{lem:1_contaminated_vtx}, so suppose $i \geq 2$.  By Lemma \ref{lem:reduction_overall} the cops may makes a series of moves to either win or ensure that at the end of some subsequent probing phase, the contamination will be contained in $N(u) \cup N(v)$ for some non-adjacent point vertex $u$ and line vertex $v$, and one partite set will contain at most $i-1$ contaminated vertices.  Through repeating this process, by Lemma \ref{lem:reduction_overall}, the cops will either win or ensure that at the end of some subsequent probing phase, the contamination will be contained in $N(u) \cup N(v)$ for some non-adjacent point vertex $u$ and line vertex $v$, and one partite set will contain at most $1$ contaminated vertex.  From that point in the game, the cops have a winning strategy by Lemma \ref{lem:1_contaminated_vtx}. Therefore $\dirloc(G) \le 2$ and by extension $\dirloccomp(G) \le 2$, completing the proof.
\end{proof}

\section{Open Questions}\label{sec:dirloc_future_work}

We conclude the paper with several intriguing open questions and avenues for future research.

\medskip

Corollary \ref{cor:hypercube} shows that $\dirloc(Q_n)$ is either $n$ or $n+1$; it would be nice to determine the exact value.
\begin{question}
What is the value of $\dirloc(Q_n)$?
\end{question}
We suspect that in fact $\dirloc(Q_n) = n$ for all $n$, but we have no strong evidence for this.  Recall that the bound $\dirloc(Q_n) \le n+1$ was a consequence of Theorem \ref{thm:degen_upper}, which states that if $G$ has degeneracy $k$, then $\dirloc(G) \le k+1$.  We know that this bound can be tight when $k=2$ (and that it cannot be tight when $k=1$), but perhaps it can be improved when $k \ge 3$.
\begin{question}
Is the bound in Theorem \ref{thm:degen_upper} ever tight for graphs of degeneracy 3 or greater?
\end{question}
Another approach toward improving the upper bound on $\dirloc(Q_n)$ would be to resolve the following question:
\begin{question}
Is it always true that $\dirloc(G) \le \Delta(G)$?
\end{question}
Note that the degeneracy of a graph $G$ is always bounded above by $\Delta(G)-1$ except when $G$ is regular; hence, one need only consider regular graphs, since otherwise the bound follows from Theorem \ref{thm:degen_upper}.

\medskip

Finally, we conclude with what is, in our opinion, the most pressing and intriguing open question about the directional localization game.  Our work in the present paper focused primarily on the partial-feedback model of the game; while some upper bounds on $\dirloccomp$ were obtained as consequences of bounds on $\dirloc$, only a few proofs actually analyzed the full-feedback model directly.  There is a great deal that we do not know about $\dirloccomp$.  The eagle-eyed reader may have noticed that we did not provide any examples of graphs having ``large'' values of $\dirloccomp$.  There is good reason for this: namely, we don't know of any such graphs.
\begin{question}
Is there a graph $G$ such that $\dirloccomp(G) > 2$?
\end{question}

\medskip


\bibliography{references}

\end{document}